\documentclass[12pt]{amsart}
\usepackage{amsmath, amsthm, amscd, amsfonts, amssymb, latexsym, graphicx, color}

\usepackage[utf8]{inputenc}

\usepackage{amsmath,amsthm}
\usepackage{amssymb,latexsym}
\usepackage{enumerate}

\newtheorem{theorem}{Theorem}[section]
\newtheorem{lemma}[theorem]{Lemma}
\newtheorem{proposition}[theorem]{Proposition}
\newtheorem{corollary}[theorem]{Corollary}
\theoremstyle{definition}

\theoremstyle{remark}
\newtheorem{remark}[theorem]{Remark}
\numberwithin{equation}{section}

\begin{document}

\title[Stratifications and Yomdin--Gromov parametrizations]
      {Strong stratifications and uniform \\
       Yomdin--Gromov parametrizations \\
       in valued fields with analytic structure}

\author[Krzysztof Jan Nowak]{Krzysztof Jan Nowak}




\date{}

\begin{abstract}
The paper concerns uniform Yomdin--Gromov para\-me\-trizations together with an estimate of their number, which generalizes a theorem by Cluckers--Forey--Loeser to arbitrary equi\-characteristic zero valued fields with analytic structure.
To this end, we establish a certain strong analytic stratification of definable sets, based on a definable non-Archimedean version of Bierstone--Milman's canonical desingulari\-zation algorithm from our earlier paper. Other basic tools are: elimination of valued field quantifiers, term description of functions definable in analytic structures, and Lipschitz cell decomposition compatible with $RV$-parametrized sets in Hensel minimal structures.
\end{abstract}

\maketitle

\section{Introduction}

We begin by highlighting the motivation and stating the main results. The introduction of basic concepts and notation will be postponed until Section~2.
This article is largely inspired by the problem of the existence of Yomdin--Gromov $T_{r}$-parametrizations, together with an estimate of their number, of sets definable over valued fields $K$ with analytic structure. A brief historical background, in real geometry and in non-Archimedean geometry, will be provided in Section 6.

\vspace{1ex}

Partial results in this direction were achieved in the papers~\cite{C-Com-L,C-Fo-L}. Among others, Cluckers--Forey--Loeser achieved \cite[Theorem~3.1.5]{C-Fo-L} on uniform parametrizations for models $K$ of the algebraic Skolemization (in the residue field sort) $T$ of the theory of equicharacteristic zero, complete, discretely valued fields, possibly augmented by adding some restricted analytic function symbols coming from~\cite[Example~4.4(1)]{C-Lip-0}.

\vspace{1ex}

The basic idea of their proof is first to achieve a certain strong $T_{1}$-parametrization defined on an open cell $P \subset K^{m}$, with zero centers, by terms satisfying a condition $(*)$. Such a parametrization is globally analytic on each box $B \subset K^{m}$ of $P$, and even on each associated box $B_{as} \subset (K^{alg})^{m}$; here
$K^{alg}$ denotes the algebraic closure of $K$. This allows them to obtain a $T_{r}$-parametrization via precomposing by the $r$-th power functions (\cite[Remark~3.1.6]{C-Fo-L}). The passage to the algebraic closure, in turn, enables the use of Cauchy's estimates, which is crucial for the calculations related to power substitution. The philosophy behind this approach is as follows:

\vspace{1ex}

$\bullet$ firstly, precomposing a $T_{1}$-parametrization with the $r$-th power functions yields $T_{r}$-property on boxes;

$\bullet$ secondly, global $T_{1}$-property and $T_{r}$-property on boxes imply global $T_{r}$-property.

\vspace{1ex}

The idea of composing with $r$-th power functions was earlier used in the real case by Cluckers--Pila--Wilkie~\cite{C-P-W}. There the difficulty lies in that o-minimal cells have frontiers which are also modified when composed with power functions. On the other hand, the absence of frontiers and of convexity arguments was a challenge leading them to the concept of $T_{r}$-maps (cf.~\cite{C-Com-L,C-Fo-L}).

\vspace{1ex}

The main purpose of our paper is to develop the theory of Yomdin--Gromov in arbitrary equicharacteristic zero valued fields with analytic structure (Theorem~\ref{strong-r}). Generally, we adopt the strategy of Cluckers--Forey--Loeser. Our goal is thus to achieve strong $T_{1}$-parametrizations (Theorem~\ref{strong-1}), and then  $T_{r}$-parametrizations via precomposing with power functions. To realize it, we shall need a term description of definable functions after a very limited Skolemization in the sort $RV$.

\vspace{1ex}

Because of Cauchy's estimates involved when substituting power functions, we shall separately treat the case where the value group is densely ordered and the opposite one, which is much more delicate.

\vspace{1ex}

In the latter case, the challenge for us was the passage to the algebraic closure of the ground field $K$ in the general context of valued fields with analytic structure. This equivalent of the complexification method in real algebraic geometry will be achieved by means of the concept of a certain strong stratification. The technique of such stratifications is based on a definable non-Archimedean version of Birstone--Milman's canonical desingularization algorithm by blowing up, achieved in our paper~\cite{Now-Sym}.

\vspace{1ex}

This paper is organized as follows. Section~2 provides some necessary concepts and notation concerning valued fields with analytic structure, along with a language $\mathcal{L}_{rv}$ for the leading term structure $RV$ from our paper~\cite{Now-APAL}. Furthermore, we discuss some basic tools which are available in the geometry of separated versus strictly convergent analytic structures, and in Hensel minimal geometry as well. Also, we remind the reader of the concepts of an ordinary cell, a parametrized cell and a twisted box of a cell, and finally the fundamental Theorem~\ref{SLCD} on Lip\-schitz cell decomposition compatible with a family of $RV$-parametrized sets.

\vspace{1ex}

Every strictly convergent Weierstrass system can be extended by definition to a separated Weierstrass system (cf.~\cite{C-Lip}). We prefer to deal with the latter because it offers more suitable quantifier elimination. On the other hand, definable resolution of singularities is then a far more subtle issue, and requires some extra efforts (see Section~3 and~\cite{Now-Sym}).

\vspace{1ex}
In Section~3, we recall the definable version of Bierstone--Milman's desingularization algorithm on strong analytic manifolds over a field $K$ with separated analytic structure. It was achieved in our earlier paper~\cite{Now-Sym}. Finally, stratifications of analytic sets are established.

\vspace{1ex}

Section~4 is devoted to Theorem~\ref{strat-strong} on strong strati\-fi\-cations of definable sets, whose proof uses our Theorem~\ref{desing-th} ff.\ on desingularization of terms and elimination of valued field quantifiers in separated analytic structures. This is a new result, which holds for sets definable over all, not necessarily algebraically closed, equicharacteristic zero valued fields with analytic structure. It provides much more precise description of the strata (cf.~Remark~\ref{emphasize}), and is the key ingredient applied in the proof of the existence of strong $T_{1}$-stratifications.

\vspace{1ex}

Section~5 provides a certain term description of definable functions, which is one of the basic tools in the proof of the existence of strong $T_{1}$-stratifications. It is available after a very limited, algebraic Skolemization, which consists only in adding the algebraic Skolem functions for the roots in the $RV$-sort, and for the polynomial equations in the residue field sort $\mathcal{R} \subset RV$.

\vspace{1ex}

Finally, the main result of this paper, i.e.\ the existence of Yomdin--Gromov parametrizations, together with an estimate of their number, is proven in Section~6.

\vspace{1ex}

We conclude the introduction with the following comment. It seems that Hensel minimal structures form a natural habitat for Yomdin--Gromov parametrizations, which seems to be a problem still open. To our best knowledge, the only results achieved in the general Hensel minimal settings, are those from the papers~\cite{C-H-R-1,Ver} covering merely the case of definable curves (i.e.\ definable sets of dimension one).

\section{Basic notions and tools}

We recall some basic terminology concerning analytic structures on valued fields from~\cite{C-Lip-0,C-Lip}.
Following the convention from~\cite{BGR,LR-separated,LR-subanalytic,LR-strat,LR-uniform,C-Lip-R,C-Lip-0,C-Lip}, the valuation ring of the valued field $K$ and its maximal ideal are denoted by $K^{\circ}$ and $K^{\circ\circ}$. This notation corresponds with another tradition represented by~\cite{Bas,E-P,C-Com-L,C-Fo-L,C-H-R,C-H-R-1,Ver}:
$$ K^{\circ} = \mathcal{O}_{K} \ \ \text{and} \ \ K^{\circ\circ} = \mathcal{M}_{K}. $$
We specify the language $\mathcal{L}_{rv}$ on the imaginary sort $RV$ of Henselian valued fields $(K,v)$ from~\cite[Section~2]{Now-APAL}, which goes back to Basarab~\cite{Bas}. The residue field is denoted by $Kv = \mathcal{O}_{K}/\mathcal{M}_{K}$. Let
$$ rv: K^{\times} \to G(K) := K^{\times}/(1+\mathcal{M}_{K}) $$
be the canonical group epimorphism. Since $vK \cong K^{\times}/\mathcal{O}_{K}^{\times}$, we get the canonical group epimorphism $\bar{v}: G(K) \to vK$ and the following exact sequence
\begin{equation}\label{exact}
1 \to (Kv)^{\times} \to G(K) \to vK \to 0.
\end{equation}
We put $v(0) = \infty$ and $\bar{v}(0) = \infty$. For simplicity, we shall write
$$ v(a) = (v(a_{1}),\ldots,v(a_{n})) \ \ \text{or} \ \ rv(a) = (rv(a_{1}),\ldots,rv(a_{n})) $$
for an $n$-tuple $a = (a_{1},\ldots,a_{n}) \in K^{n}$.

\vspace{1ex}

The language on the auxiliary imaginary sort
$$ RV(K) := G(K) \cup \{ 0 \} $$
is the language of groups $(1, \cdot)$ with one unary predicate $\mathcal{P}$ such that
$$ \mathcal{P}_{K}(\xi) \ \Leftrightarrow \ \bar{v}(\xi) \geq 0; $$
here we put $\xi \cdot 0 = 0$ for all $\xi \in RV(K)$. The predicate
$$ \mathcal{R}(\xi) \ \ \Longleftrightarrow \ \  [\xi=0 \ \vee \  (\xi \neq 0 \wedge \mathcal{P}(\xi) \wedge \mathcal{P}(1/\xi))] $$
will be construed as the residue field $Kv$ with the language of rings $(0,1,+,\cdot)$; obviously,
$\mathcal{R}_{K}(\xi) \ \Leftrightarrow \ \bar{v}(\xi) =0$. The sort $RV$ binds together the residue field and value group. We have one connecting map
$$ rv: K \to RV(K), \ \ rv(0) = 0. $$

\vspace{1ex}

The valuation ring can be defined by putting $\mathcal{O}_{K} = rv^{-1}(\mathcal{P}_{K})$.
The residue map $r: \mathcal{O}_{K} \to Kv$ will be identified with the map
$$ r(x) = \left\{ \begin{array}{cl}
                        rv(x) & \mbox{ if } \ x \in \mathcal{O}_{K}^{\times}, \\
                        0 & \mbox{ if } \ x \in \mathcal{M}_{K}.
                        \end{array}
               \right.
$$

Now let
$$ \mathcal{A} = \{ A_{m,n} \}_{m,n \in \mathbb{N}} $$
be a separated Weierstrass $(A,I)$-system, where $I$ is a proper ideal of a commutative ring $A$ with unit.
Fix a non-trivially valued equicharacteristic zero field $K$ with a separated analytic $\mathcal{A}$-structure, which is given by a collection
$$ \alpha = \{ \alpha_{m,n} \}_{m,n \in \mathbb{N}} $$
of homomorphisms from $A_{m,n}$ to the ring of $K^{\circ}$-valued functions on $(K^{\circ})^{m} \times (K^{\circ \circ})^{n}$.

\vspace{1ex}

The ground field $K$ will be considered with the analytic, two sorted language $\mathcal{L}_{\mathcal{A}}$ specified as follows:

\vspace{1ex}

$\bullet$ on the main, valued field sort $K$, the language of rings $(0,1,+,-,\cdot)$ with the multiplicative inverse $(\cdot)^{-1}$, with convention $0^{-1}=0$, and with the names of all functions of the collection $\mathcal{A}$;

$\bullet$ on the auxiliary sort $RV(K)$, the language $\mathcal{L}_{rv}$ recalled at the beginning of this section.

\vspace{1ex}

Power series from $A_{m,n}$ are construed via $\alpha$ as $f^{\alpha}=\alpha_{m,n}(f)$ on their natural domains, and as zero outside them.

\vspace{1ex}

Note that the definable sets in the auxiliary sort $RV$ are precisely those already definable in the pure valued field language. Denote by $T_{\mathcal{A}}$ the $\mathcal{L}_{\mathcal{A}}$-theory of all Henselian, non-trivially valued fields of equicharacteristic zero with analytic $\mathcal{A}$-structure. This theory is $\omega$-h-minimal (cf.~\cite[Theoem~6.2.1]{C-H-R}) and, a fortiori, 1-h-minimal. Every model of $T_{\mathcal{A}}$ is a Henselian field (cf.~\cite[Proposition~4.5.10]{C-Lip-0}).

\vspace{1ex}

Without changing the family of definable sets, we can assume, by taking the quotient with the kernel of $\alpha_{0,0}$, that the homomorphism $\alpha_{0,0}$ from $A_{0,0}$ into $K^{\circ}$ is injective; then so are all the homomorphisms $\alpha_{m,n}$. By extension of parameters, one can adjoin the elements of $K^{\circ}$ to $A_{0,0}$ to get a separated Weierstrass system $\mathcal{A}(K)$ over $(K^{\circ},K^{\circ \circ})$. Then $K$ has a unique analytic $\mathcal{A}(K)$-structure (cf.~\cite[Section~4.5]{C-Lip-0}).

\vspace{1ex}

We shall for simplicity identify the analytic power series
$$ f = f(\xi,\rho) \in A_{m,n}^{\dag} := K \otimes_{K^{\circ}} A_{m,n} $$
with their interpretations $f^{\alpha}$ on their natural domains.

\vspace{1ex}

Henceforth, by $\mathcal{L}$-definable or $\mathcal{L}_{A}$-definable we mean 0-definable or $A$-definable for a set of parameters $A$.


\vspace{1ex}

The theory $T_{\mathcal{A}}$, unlike that of strictly convergent analytic structures, admits elimination of valued field quantifiers in the language $\mathcal{L}_{\mathcal{A}}$ (cf.~\cite[Theorem~6.3.7]{C-Lip-0}).

\vspace{1ex}

The rings $A_{m,0}^{\dagger}(K)$ and $A_{0,n}^{\dagger}(K)$ of power series with only one kind of variables enjoy --- similarly as in the classical rigid analytic geometry --- good algebraic properties such as to be Noetherian, factorial, normal or excellent, because they fall under the Weierstrass--R\"{u}ckert theory (cf.~\cite[Proposition~5.2.5]{C-Lip-0} and~\cite[Section~5.2]{BGR}). Yet it seems (personal communication with the authors of~\cite{C-Lip-0}) that the rings $A_{m,n}^{\dagger}(K)$ of globally analytic functions with two kinds of variables defined on the spaces
$$ (K^{\circ})^{m} \times (K^{\circ \circ})^{n}, \ \ m,n \in \mathbb{N}, $$
may fail to satisfy good algebraic properties such as to be Noetherian or excellent. Therefore the techniques of resolution of singularities, as e.g.\ the ones by Bierstone--Milman~\cite{BM} or Temkin~\cite{Tem-2}, cannot be directly applied to those spaces.

\vspace{1ex}

The opposite situation occurs in the case of strictly convergent analytic structures. Then the rings of globally analytic functions are Noetherian and excellent because R\"{u}ckert's theory applies to them (cf.~\cite{C-Lip} and~\cite[Proposition~5.2.5]{C-Lip-0}).
But elimination of valued field quantifiers requires to augment the language with the henselian functions (cf.~\cite[Theorem~3.4.4]{C-Lip}).

\vspace{1ex}

To make both these powerful tools of analytic geometry available at the same time, we gave in our paper~\cite{Now-Sym} a definable, non-Archimedean adaptation of the canonical desingularization by blowing up (hypersurface case) due to Bierstone--Milman~\cite{BM}. It was carried out within a category of definable, strong analytic manifolds and maps, which is more flexible than that of affinoid varieties and maps.

\vspace{1ex}

Transformations to normal crossings and resolution of singularities by blowing up can be applied to various geometrical and topological problems over non-locally compact fields via our closedness theorem (cf.~\cite[Theorem~1.3]{Now-APAL} and also~\cite{K-N,Now-Sel,Now-Sing}). This theorem holds for all Hensel minimal structures on equicharacteristic zero fields $K$ which satisfy the following condition:

\vspace{1ex}

{\em Every definable subset in the auxiliary sort $RV(K)$ is already definable in the pure algebraic language.}

\vspace{1ex}

Observe that every finite definable open covering of a definable closed subset of $K^{n}$ has a finite open refinement that is a partition into clopen subsets, because the affine space $K^{n}$ (and even every definable Hausdorff LC-space in a Hensel minimal structure) is definably ultraparacompact (cf.~\cite[Corollary~7.7]{Now-APAL}).

\vspace{1ex}

Now, we remind the reader of the concept of an ordinary and of a parametrized cell.
For $m \leq n$, denote by $\pi_{\leq m}$ or $\pi_{< m+1}$ the projection $K^{n} \to K^{m}$ onto the first $m$ coordinates; put $x_{\leq m} = \pi_{\leq m}(x)$. Let $C \subset K^{n}$ be a non-empty 0-definable set, $j_{i} \in \{ 0, 1 \}$ and
$$ c_{i} : \pi_{<i}(C) \to K, \ \ \ i=1,\ldots,n, $$
be 0-definable functions. Then $C$ is called an (ordinary) 0-definable cell with center tuple $c = (c_{i})_{i=1}^{n}$ and of cell-type $j =(j_{i})_{i=1}^{n}$ if it is of the form:
$$ C = \left\{ x \in K^{n}: (rv(x_{i} - c_{i}(x_{<i})))_{i=1}^{n} \in R \right\} $$
for a (necessarily 0-definable) set
$$ R \subset \prod_{i=1}^{n} \, j_{i} \cdot G(K), $$
where $0 \cdot G(K) = {0} \subset RV(K)$ and $1 \cdot G(K) = G(K) \subset RV(K)$. One can similarly define $A$-definable cells.

\vspace{1ex}

For any $r \in R$, the subset
$$ \left\{ x \in K^{n}: (rv(x_{i} - c_{i}(x_{<i})))_{i=1}^{n} = r \right\} $$
is called a twisted box of the cell $C$.

\vspace{1ex}

If the algebraic and definable closure do not coincide, parameterized cells are required. They consist of infinitely many ordinary cells with parameters running over the leading term sort $RV$.

\vspace{1ex}

More precisely, consider a 0-definable function $\sigma: C \to RV(K)^{k}$. Then $(C,\sigma)$ is called a 0-definable parameterized (by $\sigma$) cell if each set $\sigma^{-1}(\xi)$,  $\xi \in \sigma(C)$, is a $\xi$-definable cell with some center tuple $c_{\xi}$ depending definably on $\xi$ and of cell-type independent of $\xi$.
If the language $\mathcal{L}$ has an angular component map $\overline{ac}: K \to Kv$ (called sometimes, a coefficient map), then one can take $\sigma$ in the above definition to be residue field valued instead of $RV$-valued (cf.~\cite{C-H-R}, Theorem~5.7.3).

\vspace{1ex}

In the sequel, we shall need the following fundamental theorem on Lipschitz cell decomposition with preparation of $RV$-parametrized sets (cf.~\cite[Theorem~5.2.4]{C-H-R} with Addenda~1 and~5), which holds in any 1-h-minimal structure.

\begin{theorem}\label{SLCD}
For 0-definable sets
$$ X \subset K^{n} \ \ \text{and} \ \ P \subset X \times RV(K)^{t}, $$
there exists a finite decomposition of $X$ into 0-definable parametrized cells $(C_{k},\sigma_{k})$ such that the fibers of $P$ over each twisted box of each $C_{k}$ are constant or, equivalently, the fiber of $P$ over each $\xi \in RV(K)^{t}$ is a union of some twisted boxes from the cells $C_{k}$.

Furthermore, one can require that each $C_{k}$ is, after some coordinate permutation, a parametrized cell of type $(1,\ldots,1,0,\ldots,0)$ with 1-Lipschitz centers
$$ c_{\xi} = (c_{\xi,1},\ldots,c_{\xi,n}), \ \ \xi \in \sigma(C). $$
Such cells $C_{k}$ shall be called 0-definable parametrized Lipschitz cells.
\end{theorem}

\section{Stratification of analytic sets}

In the paper~\cite{Now-Sym}, we gave an $\mathcal{L}_{\mathcal{A}}$-definable desingularization algorithm on strong analytic manifolds $M=M_{0}$ over a field $K$ with separated analytic structure. Strong analyticity is a model-theoretic strengthening of the weak concept of analyticity (treated in the classical case e.g.~by Serre~\cite{Se}) determined by a separated Weierstrass system, which works well within the definable settings.

\vspace{1ex}

We first recall some basic notions from~\cite{Now-Sym} related to the desingularization algorithm, and then shall introduce three global object, namely T-analytic manifolds, functions and subdomains (T after Tate). The concept of a T-analytic subdomain may be regarded as an analogue of that of an R-subdomain from quasi-affinoid geometry from~\cite{LR-separated}.

\vspace{1ex}

Strong analytic functions and manifolds are those analytic and definable in the structure $K$, which remain analytic in each field $L$ elementarily equivalent to $K$ in the language $\mathcal{L}_{\mathcal{A}}$. For instance, strong analytic manifolds are those obtained by means of the implicit function theorem, and the zero loci of strong analytic submersions.

\vspace{1ex}

The algorithm under consideration is used to simultaneously transform a finite number of strong analytic functions
$$ g_{i}:M \to K, \ \ i=1,\ldots,p, $$
to normal crossings, or to resolve the singularities of the hypersurface $X$ determined (in the algebraic sense) by a strong analytic function $g: M \to K$. It is a definable, non-Archimedean adaptation of the canonical desingularization by blowing up (hypersurface case) due to Bierstone--Milman~\cite{BM}. It consists of blowing up successively each mani\-fold $M_{j}$:
$$ \sigma_{j+1}: M_{j+1} \to M_{j}, \ \ j=0,\ldots,s, $$
along an admissible smooth center $C_{j}$; denote by
$$ \sigma := \sigma_{1} \circ \ldots \circ \sigma_{s}: M_{s} \to M_{0} $$
the composite of all the $\sigma_{j}$. Eventually, the pullbacks $g_{i} \circ \sigma$ are normal crossing divisors, the final transform $X'$ of $X$ under $\sigma$ is smooth (unless empty), and $X'$ and the final exceptional divisor $E'$ of $\sigma$, simultaneously have only normal crossings. The successive admissible centers are the maximum strata of a finitary invariant.


\begin{remark}\label{desing-uniform}
Observe that our $\mathcal{L}_{\mathcal{A}}$-definable desingularization algorithm works uniformly in models $L$ of the theory $T_{\mathcal{A}}$. Indeed, the analytic $\mathcal{A}$-structure on $K$ has a unique extension to an analytic $\mathcal{A}$-structure on the algebraic closure $K_{alg}$ of $K$ (cf.~\cite[Theorem~4.5.11]{C-Lip-0}). It follows from the canonical character of the algorithm that its output over $K$ is the restriction of that on $K_{alg}$. The theory of algebraically closed, non-trivially valued fields with analytic $\mathcal{A}$-structure admits quantifier elimination in the language $(0,1,+,-,\cdot, (\cdot)^{-1}, \dotplus, v)$, where $\dotplus$ denotes the operation in the value group and $v$ the valuation operation (cf.\cite[Theorem~4.5.15]{C-Lip-0} and~\cite[Theorem~4.2]{LR-uniform}). Therefore $K_{alg}$ is a prime model of this theory. Hence, by the canonical character of desingularization algorithm, its outputs over $L_{alg}$ and $K_{alg}$ are given by the same $\mathcal{L}_{\mathcal{A}}$-formula. The assertion thus follows since the output over $L$ is the restriction of that over $L_{alg}$.
\end{remark}

\begin{remark}\label{SNC}
For desingularization of terms, we shall need a stronger version of simultaneous transformation to normal crossings to the effect that locally the pullbacks
$$ g_{i}^{\sigma} := g_{i} \circ \sigma: M_{s} \to K $$
are normal crossing divisors on $M_{s}$ linearly ordered with respect to the divisibility relation. This can be obtained by a routine application of \cite[Lemma~4.7]{BM-0}.
\end{remark}


In this paper, we are interested in the desingularization of $\mathcal{L}_{\mathcal{A}}$-terms; in particular, in the desingularization of globally analytic functions $f \in A_{m,n}^{\dag}(K)$ on the spaces
$$ M_{0} = (K^{\circ})^{m} \times (K^{\circ \circ})^{n}. $$

\vspace{1ex}

The rings $A_{0,m+n}^{\dagger}(K)$  of globally analytic functions with one kind of variables have good algebraic properties; excellence, in particular. Hence the canonical desingularization process on the open unit polydisc $(K^{\circ \circ})^{m+n}$ (and likewise on its translates in $M_{0}$) outputs certain global algebro-analytic objects. All successive admissible centers $C_{j}$ and manifolds $M_{j}$ are contained in the product $M_{0} \times \mathbb{P}^{N}(K)$ with some positive integer $N$, and over $(K^{\circ \circ})^{m+n}$ (and likewise over its translates) are defined by finitely many functions from
$$ A_{0,m+n}^{\dagger}(K)[y_{0},\ldots,y_{N}] $$
being homogeneous forms with respect to the homogeneous co\-ordinates $(y_{0}:y_{1}: \ldots :y_{N})$ in $\mathbb{P}^{N}(K)$.

\vspace{1ex}

In the sequel, we denote T-analytic manifolds for the globally admissible centers $C_{j}$ and manifolds $M_{j}$ obtained in the desingularization algorithm. Further, globally
analytic functions on such $M_{j}$ or $C_{j}$ are defined as the pull-backs to $M_{j}$ of globally analytic functions $A_{m,n}^{\dagger}(K)$ on $M_{0}$, or, respectively, the restrictions of these pull-backs to $C_{j}$.

\vspace{1ex}

We define globally T-analytic functions and T-subdomains recursively as follows. Starting with any pair of globally analytic functions $f,g$ on $M_{0}$ and a suitable transformation $\sigma: M_{1} \to M_{0}$ to normal crossings from Remark~\ref{SNC}, the map
$$ (f^{\sigma} : g^{\sigma}): M_{1} \to \mathbb{P}^{1}(K) $$
is a well defined, globally analytic function. By T-subdomains $U$, at the first level of complexity, we mean finite intersections of clopen sets of the form
$$ \{ x \in M_{1}: \ |(f^{\sigma}(x) : g^{\sigma}(x))| \leq 1 \}, \ \ \{ x \in M_{1}: \ |(f^{\sigma}(x) : g^{\sigma}(x))| < 1 \}; $$
and by T-analytic funcions on $U$ we mean the restriction to $U$ of such globally analytic functions $(f^{\sigma} : g^{\sigma})$. We continue the process with any pair of T-analytic functions on a T-subdomain of the T-analytic manifold $M_{1}$, thus increasing the level of complexity.
The concept of a T-subdomain is in a sense an analogue of that of an R-subdomain from~\cite[Section~5.3]{LR-separated}, which in turn generalizes by iteration that of a quasi-rational subdomain from rigid analytic geometry. Unlike rational subdomains (cf.~\cite[Section~7.2.2]{BGR}), the complement of a T-subdomain is a finite union of T-subdomains.

\begin{remark}\label{subdomains}
T-subdomains $U$ and T-analytic functions on $U$ will be used in Section~4 in the analysis of terms with the restricted divisions $D_{0}$ and $D_{1}$.
\end{remark}

\begin{remark}
Definable desingularization on the polydiscs $(K^{\circ \circ})^{m+n}$ and their translates, along with its canonical character, can be seen from the perspective of Beth's definability theorem.
\end{remark}

Via definable resolution of singularities, we obtain the following result on stratification of analytic sets.


\begin{proposition}\label{strat-anal}
Let $X$ be a strong analytic (T-analytic, respectively) subset of a strong analytic (T-analytic) manifold $M$, i.e.\ the zero set
$$ X := \{ x \in M: f_{1}(x) = \ldots = f_{k}(x) =0 \} $$
of a finite number of strong analytic (T-analytic) functions $f_{1},\ldots, f_{k}$ on $M$. Then $X$ is a finite disjoint union of subanalytic submanifolds $V_{i}$ each of which is a one-to-one projection of a strong semianalytic (and T-analytic) submanifold $W_{i}$ of the product $M \times \mathbb{P}^{N}(K)$ for some positive integer $N$. Those one-to-one projections are restrictions of some multi-blowups with admissible smooth centers.
\end{proposition}

\begin{proof}
Consider first the case $k=1$ and resolve singularities of the hypersurface $X = X_{1}$. Adopt the notation of Section~2 and put
$$ \sigma_{\leq j} := \sigma_{1} \circ \sigma_{2} \circ \ldots \circ \sigma_{j}, \ \  j=2,3,\ldots,s. $$
Let $E_{j}$ denote the set of exceptional hypersurfaces in $M_{j}$ of the multi-blowup $\sigma_{\leq j}$. Notice that $C_{j}$ and $E_{j}$ simultaneously have only normal crossings (cf.~\cite{BM,Now-Sym}). Then $X$ is the disjoint union of the sets
$$ C_{0} \cup \sigma_{1}(C_{1} \setminus E_{1}) \cup \sigma_{\leq 2}(C_{2} \setminus E_{2}) \cup \ldots \cup \sigma_{\leq s-1}(C_{s-1} \setminus E_{s-1}) \cup
   \sigma_{s}(\widetilde{X} \setminus E_{s}). $$
Thus we are done because the restriction of $\sigma_{\leq j}$ to the complement of $E_{j}$ is one-to-one.

\vspace{1ex}

Further proceed with induction on $k$. For instance, consider the case $k=2$, and two analytic hypersurfaces $X = X_{1}$ and $Y = X_{2}$ defined by two analytic functions $f_{1}$ and $f_{2}$, respectively. Then $X \cap Y$ is the disjoint union of the sets
$$ (C_{0} \cap Y) \cup \sigma_{1}((C_{1} \cap Y) \setminus E_{1}) \cup \sigma_{\leq 2}((C_{2} \cap Y) \setminus E_{2}) \cup \; \ldots  $$
$$ \ldots \; \cup \sigma_{\leq s-1}((C_{s-1} \cap Y) \setminus E_{s-1}) \cup \sigma_{s}((\widetilde{X} \cap Y) \setminus E_{s}). $$
Now resolve singularities of each hypersurface $C_{j} \cap Y$ defined on the analytic manifold $C_{j}$ by the restriction of $f_{2}$, and proceed in the same way as for $k=1$.
\end{proof}

\section{Term desingularization and stratification of sets}

The main aim of this section is to achieve a certain strong stratification of definable sets, whose proof is based on a simultaneous desingularization of terms and on elimination of valued field quantifiers.
The purpose of simultaneous desingularization of a finite number of $\mathcal{L}_{\mathcal{A}}$-terms $t_{1}(x),\ldots,t_{p}(x)$ is, in turn, to transform them into T-analytic functions through one-to-one strong analytic maps defined on a finite number of T-analytic manifolds coming from resolution of singularities.

\vspace{1ex}

We begin with some preparatory observations regarding the necessary concepts involved in our approach.

\vspace{1ex}

Every $\mathcal{L}_{\mathcal{A}}$-term $t(x)$ is a nested composition of analytic functions from the system $\mathcal{A}$, the multiplicative inverse and polynomials over $K$. However, it is more convenient for desingularization to consider the valued ring $K^{\circ}$ with distinguished maximal ideal $K^{\circ \circ}$ as an auxiliary sort, and to replace the multiplicative inverse $(\cdot)^{-1}$ with two restricted divisions
$$ D_{0}: K^{\circ} \times K^{\circ} \to K^{\circ} \ \ \text{and} \ \ D_{1}: K^{\circ \circ} \times K^{\circ \circ} \to K^{\circ \circ} $$
defined by the formulae
$$ D_{0}(x,y) := \left\{ \begin{array}{cl}
                        x/y & \mbox{ if } \ |x| \leq |y| \neq 0 \\
                        0 & \mbox{ otherwise } \
                        \end{array}
               \right.
$$
and
$$ D_{1}(x,y) := \left\{ \begin{array}{cl}
                        x/y & \mbox{ if } \ |x| < |y| \\
                        0 & \mbox{ otherwise. } \
                        \end{array}
               \right.
$$
The first division is a term of the valued ring sort and the other is a term of the maximal ideal sort. Two restricted divisions are needed because of two kinds of variables that occur in separated analytic structures.

\vspace{1ex}

Denote by $\mathcal{L}_{\mathcal{A}}^{D}$ the analytic language $\mathcal{L}_{\mathcal{A}}$ with the unary function symbol $(\cdot)^{-1}$ replaced with the two binary function symbols $D_{0}$ and $D_{1}$. An $\mathcal{L}_{\mathcal{A}}^{D}$-term of either of these sorts is referred to as a $D$-function. The intended interpretation of the languages $\mathcal{L}_{\mathcal{A}}$ and $\mathcal{L}_{\mathcal{A}}^{D}$ are the valued field $K$ and the valued ring $K^{\circ}$, respectively. Notice that $D$-functions consists only of analytic functions and restricted divisions. Indeed, algebraic operations are performed on elements from the ring $K^{\circ}$ or its maximal ideal $K^{\circ \circ}$, thus polynomials can be treated as analytic functions from the Weierstrass system.

\vspace{1ex}

The reduction of the analysis of $\mathcal{L}_{\mathcal{A}}$-terms to $\mathcal{L}_{\mathcal{A}}^{D}$-terms is allowed because:

\vspace{1ex}

$\bullet$ firstly, $K^{n}$ is a disjoint union of Cartesian products of some discs $|x_{i}| \leq 1$ and annuli $|x_{j}| >1$ which are mapped into discs $|x_{j}| \leq 1$ by the multiplicative inverse $(x_{j})^{-1}$;

$\bullet$ secondly, every $\mathcal{L}_{\mathcal{A}}$-term on $(K^{\circ})^{n}$ with valuation at most $1$ is equivalent to an $\mathcal{L}_{\mathcal{A}}^{D}$-term.

\vspace{1ex}

The latter fact follows easily by induction on the number of occurrences of function symbols from the system $\mathcal{A}$. Indeed, given some terms $t(x)$, $t_{1}(x)$ and $t_{2}(x)$, one can combine them piecewise, making use of the characteristic functions of the sets defined by atomic formulae of the form:
$$ t(x) = 0, \ \ t(x) \neq 0, \ \ |t_{1}| \leq |t_{2}(x)|, \ \ |t_{1}| < |t_{2}(x)|. $$
To this end, one can use the two characteristic functions:
$$ \chi(\{ x \in K^{0}: \ x \neq 0 \}) = D_{0}(x,x) $$
and
$$ \chi(\{ (x,y) \in (K^{0})^{2}: \ 0 \neq |x| \leq |y| \}) = D_{0}(D_{0}(x,y),D_{0}(x,y)). $$

\vspace{1ex}

\begin{remark}\label{QE}
Likewise, any subset of the closed unit polydisc $(K^{\circ})^{n}$ is quantifier-free $\mathcal{L}_{\mathcal{A}}$-definable if and only if it is quantifier-free $\mathcal{L}_{\mathcal{A}}^{D}$-definable (cf.~\cite[Remark~3.2.12]{C-Lip}).
\end{remark}

At this stage, we can readily achieve simultaneous desingularization of terms.

\begin{theorem}\label{desing-th}
Consider a finite number of $D$-functions $t_{1}(x),\ldots, t_{p}(x)$ on $(K^{\circ})^{n}$. Then $(K^{\circ})^{n}$ is a finite disjoint union of strong analytic, subanalytic manifolds $S_{k} := \sigma_{k}(V_{k})$, $k=1,\ldots,s$, such that:

1) Each $\sigma_{k}: N_{k} \to (K^{\circ})^{n}$ is a multi-blowup with exceptional divisor $E_{k}$ from the transformation to normal crossings on $(K^{\circ})^{n}$ or on one of its admissible smooth centers $C$ (as in Proposition~\ref{strat-anal}). We can assume that each $N_{k}$ is a T-analytic submanifolds of $(K^{\circ})^{n} \times \mathbb{P}^{N}(K)$ for some integer $N$ large enough, and that $\sigma_{k}$ is the projection onto $(K^{\circ})^{n}$.

2) Each $U_{k} \subset N_{k}$ is a clopen T-subdomain of $N_{k}$ or of the pre-image of $C$, and $V_{k} := U_{k} \setminus E_{k}$.

3) Finally, each pull-back
$$ t_{i}^{\sigma_{k}} = t_{i} \circ \sigma_{k}, \ \ i=1,\ldots,p, \ k=1,\ldots,s $$
is the restriction to $V_{k}$ of a T-analytic function $\omega_{i,k}$ on $U_{k}$.
  \hspace*{\fill} $\Box$
\end{theorem}

\begin{proof}
We shall proceed with double induction on the number of occurrences of restricted divisions $D_{0}$ and $D_{1}$ in the $\mathcal{L}_{\mathcal{A}}^{D}$-terms under study, and on the dimension of the ambient T-analytic manifold.

\vspace{1ex}

Desingularization will consist of alternate partitioning the domain of the $D$-functions already obtained, and blowing up the resulting parts along definable smooth centers coming from transformation to normal crossings.

\vspace{1ex}

Two kinds of partitions of a resulting T-analytic manifold $N$ will be considered at each stage of the process. One is with respect to the admissible center $C \subset N$ of the subsequent blowup (or, equivalently, its exceptional divisor) and to the zero sets
$$ Z(g) := \{ x \in N: g(x)=0 \} $$
of the T-analytic functions $g$ occurring as divisors in the restricted divisions from the $D$-functions on $N$ in question. By Proposition~\ref{strat-anal} in extenso, we can thus assume that every such $g$ either vanishes nowhere or vanishes on any set of the partition.

\vspace{1ex}

The other one is with respect to clopen T-subdomains $U$ of a T-analytic manifold $N$, considered at a stage of the process, which are of the form
$$ U = \{ x \in N: \ |(f(x) : g(x))| \, \lozenge \, 1 \}, $$
where $f,g$ are T-analytic functions, and $\lozenge \in \{ \leq, < \}$, depending on whether $f$ substitutes for a variable varying over the closed or open unit disc.

\vspace{1ex}

Starting with $N = (K^{\circ})^{n}$, since two kinds of variables occur, we must partition the polydisc  $(K^{\circ})^{n}$ into clopen T-subdomains of the form
$$ \{ x \in (K^{\circ})^{n}: \ |x_{i}|<1, \ |x_{j}|=1, \ i \in I, \ j \in J \}, $$
where $I,J$ is a partition of $\{ 1,\ldots,n \}$ .

\vspace{1ex}

To understand better the idea behind the desingularization of $D$-functions, consider a single step of the process, which consists in resolving a single occurrence in a $D$-function $t$ of restricted divisions: $D_{0}(f,g)$ or $D_{1}(f,g)$ with T-analytic functions $f,g$ defined on a clopen T-subdomain $U$.

\vspace{1ex}

By means of the stronger version of transformation to normal crossings from Remark~\ref{SNC}, we see that the pullbacks
$$ D_{0}(f^{\sigma},g^{\sigma}) \ \ \text{and} \ \ D_{1}(f^{\sigma},g^{\sigma}) $$
under a suitable multi-blowup $\sigma: \widetilde{U} \to U$ are T-analytic (Remark~\ref{subdomains}) on the clopen T-subdomains of $\widetilde{U}$:
$$ \{ y \in \widetilde{U}: \ |(f(y) : g(y))| \leq 1 \} \setminus Z(g) $$
and
$$ \{ y \in \widetilde{U}: \ |(f(y) : g(y))| < 1 \}, $$
respectively.  Of course, we have to separately treat the $D$-function over the zero set $Z(g^{\sigma}) \cap \Omega$ and over the exceptional divisor $E \cap \Omega$ of the multi-blowup $\sigma$. The exceptional divisor must be treated separately since desingularization is supposed to be injective. Stratifying the T-analytic set $Z(g^{\sigma}) \cup E$ (Proposition~\ref{strat-anal}), we are done by induction on the dimension of the ambient manifold.

\vspace{1ex}

Below are five basic features of the desingularization process:

\vspace{1ex}

$\bullet$ it can be performed step by step for each term;

$\bullet$ only clopen T-subdomains $U$ of $N$ as above are blown up;

$\bullet$ partitions with lower-dimensional T-analytic subsets merely serve to stratify the T-subdomains already constructed, and enable induction on the dimension of the ambient manifold;

$\bullet$ continuing the process does not spoil the desingularization of lower complexity subterms already achieved;

$\bullet$ by the closedness theorem (cf.~\cite[Theorem~1.3]{Now-APAL}), every multi-blowup $\sigma: \widetilde{N} \to N$ is a definably closed map, and thus its restriction to the complement of the exceptional divisor $E$ is a homeomorphism onto the image which is an open subset of $N$.

\vspace{1ex}

Now, the conclusion of the theorem follows directly from the analysis of a single resolution step and the five features of the desingularization process given above.
\end{proof}

By Proposition~\ref{strat-anal}, we immediately obtain the following

\begin{corollary}\label{desing-cor}
Consider a finite number of $\mathcal{L}_{\mathcal{A}}$-terms $t_{1}(x),\ldots, t_{p}(x)$ on $(K^{\circ})^{n}$.
Then $(K^{\circ})^{n}$ is a finite disjoint union of strong analytic, subanalytic manifolds $S_{k} := \sigma_{k}(V_{k})$, $k=1,\ldots,s$, such that:

1) Each $\sigma_{k}: N_{k} \to (K^{\circ})^{n}$ is a multi-blowup with exceptional divisor $E_{k}$ from the transformation to normal crossings as in the above theorem.

2) Each $U_{k} \subset N_{k}$ is a clopen T-subdomain of $N_{k}$ or of the pre-image of an admissible smooth center $C$, and $V_{k} := U_{k} \setminus E_{k}$.

3) Finally, each pull-back
$$ t_{i}^{\sigma_{k}} = t_{i} \circ \sigma_{k}, \ \ i=1,\ldots,p, \ k=1,\ldots,s $$
is the restriction to $V_{k}$ of either a T-analytic function $\omega_{i,k}$ on $U_{k}$ or its multiplicative inverse. In the latter case, $\omega_{i,k}$ vanishes nowhere on $V_{k}$.
\end{corollary}


By elimination of valued field quantifiers (along with Remark~\ref{QE}), every
$\mathcal{L}_{\mathcal{A}}$-definable subset $X$ of $(K^{\circ})^{n}$ is a set of the form
$$ X := \left\{ x \in (K^{\circ})^{n}: (\mathrm{rv}\, t_{1}(x), \ldots, \mathrm{rv}\, t_{p}(x)) \in B \right\}, $$
where $B$ is an $\mathcal{L}_{\mathcal{A}}$-definable subset of $RV(K)^{p}$ and $t_{1},\ldots ,t_{p}$ are $D$-functions. Therefore it is not difficult to check that, moreover, every such $X$ is a finite union of sets of the form
$$ Y := \left\{ x \in (K^{\circ})^{n}: (\mathrm{rv}\, t_{1}(x), \ldots, \mathrm{rv}\, t_{p}(x)) \in B \right\}, $$
where $B$ is a definable subset of $RV(K)^{q}$ with
$$ B = A \times  \{ 0 \}^{p-q} \subset (RV(K) \setminus \{ 0 \})^{q} \times \{ 0 \}^{p-q}. $$
Then
$$ Y = \left\{ x \in (K^{\circ})^{n}: \ (\mathrm{rv}\, t_{1}(x), \ldots, \mathrm{rv}\, t_{q}(x)) \in A, \ \  \right. $$
$$ \left. t_{q+1}(x)= \ldots , = t_{p}(x) = 0 \right\}. $$
Therefore, by Theorem~\ref{desing-th}, we obtain

\begin{corollary}\label{strat-half}
Using the above assumptions and notation, every
$\mathcal{L}_{\mathcal{A}}$-definable subset $X$ of $(K^{\circ})^{n}$ is a finite disjoint union of sets of the form
$$ \sigma_{k} \left( \left\{ y \in V_{k}: \ (\mathrm{rv}\, \omega_{1,k}(y), \ldots, \mathrm{rv}\, \omega_{q,k}(y)) \in A \right\} \cap \right. $$
$$ \left. \cap \left\{ y \in V_{k}: \  \omega_{q+1,k}(y)= \ldots , = \omega_{p,k}(y) = 0 \right\} \right). $$
Further, while the first set of the intersection is an open $\mathcal{L}_{\mathcal{A}}$-definable subset $\Omega_{k}$ of $V_{k}$, the second one is a T-analytic subset of $V_{k}$ and, more precisely, the trace on $V_{k}$ of the T-analytic subset on the clopen T-subdomain $U_{k}$ defined by the T-analytic functions $\omega_{q+1,k}, \ldots , \omega_{p,k}$.
  \hspace*{\fill} $\Box$
\end{corollary}

Combining the above corollary with Proposition~\ref{strat-anal} in the version for T-analytic manifolds, we achieve the following strong version of stratifications of $\mathcal{L}_{\mathcal{A}}$-definable sets.

\begin{theorem}\label{strat-strong}
Every $\mathcal{L}_{\mathcal{A}}$-definable subset $X$ of $(K^{\circ})^{n}$ is a finite disjoint union of strong analytic manifolds $S_{k} := \sigma_{k}(W_{k})$, $k=1,\ldots,s$, such that:

1) Each $\sigma_{k}: N_{k} \to (K^{\circ})^{n}$ is a multi-blowup with exceptional divisor $E_{k}$ and each $N_{k}$ is a T-analytic submanifold of $(K^{\circ})^{n} \times \mathbb{P}^{N}(K)$ for some integer $N$ large enough;

2) For each $k=1,\ldots,s$, $W_{k} := (Z_{k} \cap \Omega_{k}) \setminus E_{k}$ is a strong analytic submanifold, where $U_{k} \subset N_{k}$ is a clopen T-subdomain of $N_{k}$ or of the pre-image of an admissible smooth center $C$ (where a new process starts via induction), $Z_{k}$ is a T-analytic subset of $U_{k}$ and $\Omega_{k}$ is an open $\mathcal{L}_{\mathcal{A}}$-definable subset of $U_{k}$.

3) The restriction $\sigma_{k}|_{W_{k}}: W_{k} \to S_{k}$ is a homeomorphic parametrization of the stratum $S_{k}$.
\end{theorem}

\begin{remark}\label{emphasize}
It seems that the only results on stratifications of sets definable over valued fields known before are the ones for subanalytic subsets of affine spaces over an algebraically closed valued field, provided in \cite[Theorem~5.2]{LR-subanalytic} ff., \cite[Theorem~4.4]{LR-strat} and \cite[Corollary~3.4.7]{C-Lip}. They state, generally speaking, that every $\mathcal{L}_{\mathcal{A}}$-subanalytic subset of $K_{alg}^{n}$ is a finite disjoint union of $\mathcal{L}_{\mathcal{A}}$-subanalytic manifolds.
Let us emphasize that our theorem on strong stratifications applies to all definable sets, not only to subanalytic ones. Moreover, the description of the strata is much more precise: every stratum is the intersection of a global $\mathcal{L}_{\mathcal{A}}$-subanalytic subset and an open $\mathcal{L}_{\mathcal{A}}$-definable subset.
\end{remark}

\begin{remark}\label{strat-uniform}
In view of Remark~\ref{desing-uniform}, the above desingularization of terms works uniformly in models $L$ of the theory $T_{\mathcal{A}}$. Therefore the desingularization process over $K$ can be treated as the restriction of the one over the algebraic closure $K_{alg}$ with multi-blowups $\sigma_{k}'$ and analytic functions $\omega_{i,k}'$. We can thus consider the sets (notation of Corollary~~\ref{strat-half}):
$$ \left\{ y \in V_{k}': \ (\mathrm{rv}\, \omega_{1,k}'(y), \ldots, \mathrm{rv}\, \omega_{q,k}'(y)) \in A \right\} \cap  $$
$$  \cap \left\{ y \in V_{k}': \  \omega_{q+1,k}'(y)= \ldots , = \omega_{p,k}'(y) = 0 \right\} , $$
which are definable over $K_{alg}$ in the $RV$-expansion of the language $\mathcal{L}_{\mathcal{A}}$ by adding $A \subset RV(K)^{q}$ as a definable subset of $RV(K_{alg})^{q}$. By the resplendency property of Hensel minimality (cf.~\cite[Theorems~4.1.19 and~6.2.1]{C-H-R}), the stratification over $K$ from the above theorem is the restriction of the one over $K_{alg}$ (with the associated data $\sigma_{k}'$, $W_{k}'$). Furthermore, each strong analytic submanifold $W_{k}$ is the trace of an $\mathcal{L}_{\mathcal{A}}$-definable T-analytic submanifold $M'$ of $(K_{alg}^{\circ})^{n} \times \mathbb{P}^{N}(K_{alg})$ on an open $\mathcal{L}_{\mathcal{A}}$-definable subset $\Omega_{k}$ of $N_{k}$. This will be used in Section~6 concerning Yomdin--Gromov parametrizations.
\end{remark}

\section{A term description of definable functions}

In this section, we give a certain term description of definable functions after a very limited, algebraic Skolemization in the sort $RV$. Here it is more convenient to denote the valuation ring and its maximal ideal by $\mathcal{O}_{K} =K^{\circ}$ and $\mathcal{M}_{K} = K^{\circ \circ}$.
The language $\mathcal{L}_{\mathcal{A}}$ to $\mathcal{L}_{\mathcal{A}}^{*}$ will be expanded by adding the following multi-sorted Henselian functions (witnesses) and roots:
$$ h_{m}: K^{m+1} \times Kv \to K \ \ \text{and} \ \ \mathrm{root}_{m}: K \times RV(K) \to K. $$

Firstly, consider the polynomial
$$ P(A,Y) := \sum_{i=0}^{m} \; A_{i}Y^{i} \in K[A,Y]; $$
if $a_{0}, \ldots, a_{m} \in \mathcal{O}_{K}$, $e \in (Kv)^{\times}$, and
$$ P(a,e) = 0 \bmod{\mathcal{M}_{K}}, \ \ \frac{\partial P}{\partial Y}(a,e) \neq 0 \bmod{\mathcal{M}_{K}}, $$
then $h_{m}(a_{0}, \ldots, a_{m},e)$ is the unique $b \in \mathcal{O}_{K}$ such that
$$ P(a,b)=0 \ \  \text{and} \ \  r(b) =rv(b) = e, $$
if such an element $b$ exists; otherwise we put $h_{m}(a_{0}, \ldots, a_{m},e)=0$.

Secondly, $\mathrm{root}_{m}(a,\xi)$ is the unique $b$ such that $b^{m} = a$ and $rv(b) = \xi$, if such an element $b$ exists; otherwise we put $\mathrm{root}_{m}(a,\xi)=0$.

\vspace{1ex}

After adding the above witness functions, definable functions have the following term structure (cf.~\cite[Section~7]{C-Lip-R}, and also~\cite[Setion~3]{C-Fo-L}):

\begin{theorem}\label{term-str}
The theory $T_{\mathcal{A}}$ has the following property uniformly in its models $K$. Let $f: X \to K$ be an $\mathcal{L}_{\mathcal{A}}$-definable function on a subset $X$ of $K^{n}$. Then there exist an $\mathcal{L}_{\mathcal{A}}$-definable function $g: X \to RV(K)^{N}$ for some positive integer $N$ and an $\mathcal{L}_{\mathcal{A}}^{*}$-term $t$ such that
$$ f(x) = t(x,g(x)) \ \ \text{for all} \ x \in X. $$
\end{theorem}


\vspace{1ex}

A similar term structure (cf.~\cite[Theorem~6.3.8]{C-Lip-0}) can be formulated by means of the following, so-called basic Henselian functions
$$ h_{m} : K^{m+1} \times RV(K) \to K, \ \ m \in \mathbb{N}, $$
whose construction (cf.~\cite[Definition~6.1.7]{C-Lip-0}) is based on the following lemma going back to Cohen~\cite{Co}, Denef~\cite{De} and Pas~\cite{Pa1}. Its proof uses a strong version of Hensel's lemma (see e.g.~\cite[Theorem~4.1.3]{E-P}) and a homothetic transformation.

\begin{lemma}
Consider the polynomial
$$ P(A,Y) = \sum_{i=0}^{m} A_{i}Y^{i} \in K[A,Y], \ \ A = (A_0,A_1,\ldots,A_m), $$
and suppose that, for $a_{i} \in K$, $i=0,\ldots,m$, and $e \in K$, $e \neq 0$, we have
$$ \min \, \{ v(a_{i}e^{i}), \ i=0,\ldots,m \} = v(a_{i_{0}}e^{i_{0}}) < v(f(a,e)) \ \ \text{for an} \ \ i_{0} \geq 1 $$
and
$$ v\left( \frac{\partial P}{\partial Y}(a,e) \right) = v(a_{i_{0}}e^{i_{0}-1}). $$
Then there is a unique $b \in K$ such that $rv(b) = rv(e)$ and $P(a,b)=0$.
  \hspace*{\fill} $\Box$
\end{lemma}

Hensel minimality enjoys the resplendency property, i.e.\ it is preserved with respect to expansions of the imaginary sort $RV$ (cf.~\cite[Theorem~4.1.19]{C-H-R}).
(One can always confine oneself to considering an expansion of $RV$ by predicates.) Hence, in particular, cell decomposition is further available after expanding the sort $RV$.
We still need a resplendent version of term structure since we shall augment the sort $RV$ by adding some algebraic Skolem functions.

\begin{remark}\label{explanation}
The reason why we present both versions of term structure is as follows. While the first version is more convenient for the proof of Proposition~\ref{descript} on piecewise description of definable functions by terms, the second one is used to define Henselian terms. The Henselian terms, built from basic Henselian functions $h_{m}$, correspond to the terms satifying condition (*) introduced by Cluckers--Forey--Loeser~\cite{C-Fo-L}, and built from Henselian witnesses and roots.
\end{remark}

\begin{proposition}\label{term-resplendent}
The above theorem on term structure remains valid after any $RV$-expansion $\mathcal{L}_{\mathcal{A}}'$ of $\mathcal{L}_{\mathcal{A}}$, i.e.\ after any expansion $\mathcal{L}_{rv}'$ of $\mathcal{L}_{rv}$. More precisely, for every $\mathcal{L}_{\mathcal{A}}'$-definable function $f: X \subset K^{n} \to K$, there exist an $\mathcal{L}_{\mathcal{A}}'$-definable function $g: X \to RV(K)^{k}$, for some $k \in \mathbb{N}$, and an $\mathcal{L}_{\mathcal{A}}^{*}$-term $t$ such that
$$ f(x) = t(x, g(x)) \ \ \text{for all} \ \ x \in X. $$
\end{proposition}

\begin{proof}
Suppose that the graph
$$ F := \mathrm{graph}\, (f) \subset K^{n}_{x} \times K_{y} $$
is defined by an $\mathcal{L}_{\mathcal{A}}'$-formula $\chi(x,y)$, which is equivalent to a finite disjunction of formulae of the form:
\begin{equation}\label{disjunction1}
  \phi(x,y) \, \wedge \, \psi(rv(p(x,y))), \ \ x = (x_1,\ldots,x_n), \ y = y_1,
\end{equation}
where $\phi$ is an $\mathcal{L}_{\mathcal{A}}$-formula in the language of the valued field sort, $\psi$ is an $\mathcal{L}_{rv}^{\prime}$-formula, and $p(x)= (p_{1}(x),\ldots,p_{r}(x))$ is a tuple of terms in the valued field sort.

It is enough to work piecewise. We can thus assume that the formula $\chi$ is of the form
$$ \phi(x,y) \, \wedge \, p_{1}(x,y) \neq 0 \, \wedge  \, \ldots \, \wedge \, p_{r}(x,y) \neq 0 \, \wedge \, \psi(rv(p(x,y))). $$
Further, we can throw away pieces of lower dimension and treat them by induction (on the dimension of the domain of $f$).

The closure $Z \in K^{n} \times K$ of the set of those points where the terms $p(x,y)$ are not continuous is $\mathcal{L}_{\mathcal{A}}$-definable of dimension $\leq n$. Then the set
$$ U := \left\{ (a,b) \in (K^{n} \times K) \setminus Z: \, p_{1}(a,b) \neq 0, \, \ldots, \, p_{r}(a,b) \neq 0, \right. $$
$$ \left. RV(K) \models \, \psi(rv(p(a,b))) \right\} $$
is an open $\mathcal{L}_{\mathcal{A}}$-definable subset of $K^{n} \times K$. Further, the set of those points $a \in K^n$, over which the fiber of $Z$ is infinite, is of dimension $< n$. As we can work piecewise, we may thus assume that the fibres of $Z$ over the points $a \in K^n$ are finite. Therefore, since
$$ F \subset ( \{ (a,b) \in K^{n} \times K: \ K \models \, \phi(a,b) \} \cap U) \cup Z, $$
we get $F \subset Y \cup Z$, where $Y$ is the $\mathcal{L}_{\mathcal{A}}$-definable set of those points
$(a,b) \in K^{n} \times K$ which are isolated in the fibre of the set
$$ \{ (a,b) \in K^{n} \times K: \ K \models \, \phi(a,b) \} $$
over $a$. Then the projection onto the first $n$ coordinates
$$ \pi_{\leq n}: Y \cup Z \to K^n $$
is finite-to-one. By parameterized cell decomposition (Theorem~\ref{SLCD}), $Y \cup Z$ is a finite disjoint union of $\mathcal{L}_{\mathcal{A}}$-definable parameterized cells $C_{i}$ which are the unions of the graphs of parameterized centers. Since $F \subset \bigcup \, C_{i}$, we can assume that $F$ is contained in one of those parameterized cells, say $C$, with an $\mathcal{L}_{\mathcal{A}}$-parametrization $\sigma: C \to RV(K)^k$ for some $k \in \mathbb{N}$. Then
$$ F \subset C = \bigcup_{\xi} \, C_{\xi}, \ \  C_{\xi} = \sigma^{-1}(\xi), $$
and each cell $C_{\xi}$ is the graph of the center $c_{\xi} : \pi_{\leq n}(C_{\xi}) \to K$.

\vspace{1ex}

Now, applying Theorem~\ref{term-str} to $\xi$-definable functions $c_{\xi}$, we obtain
$$ c_{\xi}(x) = t_{\xi}(x, d_{\xi}(x)), \ \ d_{\xi}: C_{\xi} \to RV(K)^{k(\xi)}, \ \ \xi \in \sigma(C), $$
for an $(\mathcal{L}_{\mathcal{A}}^{*} \cup \{ \xi \})$-term $t_{\xi}$ and an $(\mathcal{L}_{\mathcal{A}} \cup \{ \xi \})$-definable map $d_{\xi}$.
Further, by model-theoretic compactness, we can find a finite number of $\mathcal{L}_{\mathcal{A}}^{*}$-terms $t_{j}(x, \xi, \tau)$ and of $\mathcal{L}_{\mathcal{A}}$-definable formulae $\delta_{j}(x, \xi, \tau)$, $j=1,\ldots,s$, and a finite $\mathcal{L}_{\mathcal{A}}$-definable partition $\sigma(C) = \Sigma_{1} \cup \ldots \cup \Sigma_{s}$, which satisfy the following property:

\vspace{1ex}

\begin{em}
For every $\xi \in \Sigma_{j}$, $j=1,\ldots,s$, the formula $\delta_{j}(x, \xi, \tau)$ determines a function
$d_{j, \xi}: C_{\xi} \to RV(K)^{k_{j}}$, $k_{j} \in \mathbb{N}$, such that
$$ c_\xi(x) = t_{j}(x,\xi, d_{j,\xi}(x)) \ \ \text{for all} \ \ x \in C_{\xi}. $$
\end{em}

For each $j=1,\ldots,s$, put
$$ C_{j} := \bigcup_{\xi \in \Sigma_{j}} \, C_{\xi}, \ \ c_{j} := \bigcup_{\xi \in \Sigma_{j}} \, c_{\xi}, \ \
   d_{j} := \bigcup_{\xi \in \Sigma_{j}} \, d_{j,\xi}: C_{j} \to RV(K)^{k_{j}}. $$
Next define an $\mathcal{L}_{\mathcal{A}}$-definable function
$$ e_{j}: C_{j} \to \Sigma_{j} \subset \sigma(C) $$
by the condition $e_{j}(x)=\xi$ if $x \in C_{\xi}$.
Then
$$ c_{j}(x) = t_{j}(x,e_{j}(x),d_{j}(x)) \ \ \text{for all} \ \ x \in C_{j}. $$
Since
$$ C = \bigcup_{j=1}^{s} \, \mathrm{graph}\, (c_{j}) \ \ \text{and} \ \ F = \bigcup_{j=1}^{s} \, (F \cap \mathrm{graph}\, (c_{j})), $$
each set $F \cap \mathrm{graph}\, (c_{j})$ is the graph of the restriction of $f$ to the $\mathcal{L}_{\mathcal{A}}'$-definable set
$$ E_{j} := \pi_{\leq n}(F \cap \mathrm{graph}\, (c_{j})). $$
Hence
$$ f(x) = t_{j}(x,e_{j}(x),d_{j}(x)) \ \ \text{for all} \ \ x \in E_{j}, $$
which finishes the proof.
\end{proof}

We are now going to prove that, after adding some algebraic Skolem functions in the $RV$ sort, a piecewise description of definable functions by terms will be available.
We start with the concept of algebraic Skolem functions and algebraic Skolemization for algebraically bounded theories. For convenience, we recall some necessary terminology.

\vspace{1ex}

An $\mathcal{L}$-theory $T$ is called {\em algebraically bounded} if for every $\mathcal{L}$-formula $\phi(\bar{x},y)$ there is an integer $s_{\phi}$ such that for every model $M$ of $T$ and any tuples $\bar{a}$ from $M^{k}$ the set
$$ \{ b \in M: \ M \models \; \phi(\bar{a},b) \} $$
is either infinite or of size $\leq s_{\phi}$; or equivalently, if $Y$ eliminates the quantifier $\exists^{\infty}$, i.e.\ for every $\mathcal{L}$-formula $\phi(\bar{x},y)$ and every model $M$ of $T$ the set
$$ \{ \bar{a} \in M^{k}: \ \# \{ b \in M: \  M \models \; \phi(\bar{a},b) \} < \infty \} $$
is $\mathcal{L}$-definable. The {\em algebraic Skolem} function for an $\mathcal{L}$-formula $\phi(\bar{x},y)$ shall be the Skolem function for the formula
$$ \left[ \phi(\bar{x},y) \wedge \exists^{\leq s_{\phi}} \, z \; \phi(\bar{x},z) \right] \ \vee \ \left[ \left( \neg \exists \, z \; \phi(\bar{x},z) \;
   \vee \; \exists^{>s_{\phi}}  \, z \; \phi(\bar{x},z) \right) \wedge y = x_{1} \right],   $$
whenever the set of $\bar{x}$ variables is non empty; otherwise we add the Skolem constant for the formula
$$ \phi(y) \, \wedge \, \exists^{\leq s_{\phi}} \, z \; \phi(z). $$
Algebraic Skolemization is the operation of adding the algebraic Skolem functions to all $\mathcal{L}$-formulae.

\vspace{1ex}

We say that $T$ has algebraic Skolem functions if, in each model of $T$, the algebraic Skolem function for each $\mathcal{L}$-formula $\phi(\bar{x},y)$ can be represented by an $\mathcal{L}$-definable function. It is not difficult to check that then the model-theoretic algebraic closure coincides with the definable closure in all models of $T$. Clearly, the algebraic Skolemization of any algebraically bounded theory $T$ has algebraic Skolem functions. This also refers to algebraic Skolem functions and algebraic Skolemization in a sort of a multi-sorted structure.

\vspace{1ex}

Although the sorts $VF$ (valued field) and $VG$ (value group) are algebraically bounded, $RV$ and $RF$ (residue field) are in general not. Yet only a very limited, algebraic Skolemization in the sort $RV$ is needed for the piecewise description by terms we are aiming at. Namely it suffices to add merely the algebraic Skolem functions for the roots in the $RV$-sort, and for the polynomial equations in the residue field sort $\mathcal{R} \subset RV$. 
Denote by $\mathcal{L}_{rv}'$ and $\mathcal{L}_{\mathcal{A}}^{\dag}$ the thus extended languages $\mathcal{L}_{rv}$ and $\mathcal{L}_{\mathcal{A}}^{*}$, respectively.

\begin{proposition}\label{descript}
The theory $T_{\mathcal{A}}$ 
has the following property uniformly in its models $K$. Let $f: X \to K$ be an $\mathcal{L}_{\mathcal{A}}^{\dag}$-definable function on a subset $X$ of $K^{n}$. Then $f$ is given piecewise by $\mathcal{L}_{\mathcal{A}}^{\dag}$-terms, i.e.\ there exist a finite $\mathcal{L}_{\mathcal{A}}^{\dag}$-definable partition $X =X_{1} \cup \ldots \cup X_{s}$ and $\mathcal{L}_{\mathcal{A}}^{\dag}$-terms $t_{1},\ldots,t_{s}$ such that
$$ f(x) = t_{i}(x) \ \ \text{for all} \ \ x \in X_{i}, \ i=1,\ldots,s. $$
\end{proposition}

\begin{proof}
Since, by Proposition~\ref{term-resplendent}, Theorem~\ref{term-str} on term structure holds for the language $\mathcal{L}_{\mathcal{A}}^{\dag}$ too, we get
$f(x) = t(x,g(x))$ for an $\mathcal{L}_{\mathcal{A}}^{*}$-term $t$ and an $\mathcal{L}_{\mathcal{A}}^{\dag}$-definable function $g: X \to RV(K)^{N}$.
We are going to get rid of the function $g$ gradually, at the cost of successive finite partitioning into $\mathcal{L}_{\mathcal{A}}^{\dag}$-definable sets at each step. To illustrate this method, consider first the case where $t$ is a henselian function $h_{m}: K^{m+1} \times Kv \to K$, namely
$$ f(x) = h_{m}(x_{0},\ldots,x_{m},g_{0}(x)) \ \ \text{with} \ \ g_{0}: X \to Kv. $$
Clearly, there exist certain $m$ algebraic Skolem functions for the formula $\sum_{i=0}^{m} \lambda_{i} \xi^{i} =0$ in the sort $RF$:
$$ \omega_{j}: (Kv)^{m+1} \setminus \{ 0 \} \to Kv, \ \ j=1,\ldots,m, $$
construed by zero if necessary, whose graphs cover its locus. Then $f(x)$ is one of the values
$$ h_{m}(x_{0},\ldots,x_{m},\omega_{j}(r(x_{0}),\ldots,r(x_{m}))), \ \ j=1,\ldots,m. $$
It follows immediately that the function $f$ is given piecewise by the above $m$ terms.

\vspace{1ex}

In the above reasoning, we used the observation that the Henselian witnesses and roots involve only a finite choice in the $RV$-sort, which depends only on the canonical images in $Kv$ or $RV(K)$ of the $K$-arguments. In the case of a Henselian witness, the finite choice is determined by the polynomial which is the reduction modulo $\mathcal{M}_{K}$ of the one from the definition of that Henselian witness. Therefore, we can proceed with induction on the complexity of the $\mathcal{L}_{\mathcal{A}}^{*}$-term $t$ in order to successively get rid of the function $g$ by repeated application of the algebraic Skolem functions, as sketched above. Eventually, we come to representation of the function $f$ as given piecewise by a finite number of $\mathcal{L}_{\mathcal{A}}^{\dag}$-terms.
\end{proof}

Now, we are going to achieve decomposition of $\mathcal{L}_{\mathcal{A}}^{\dag}$-definable sets into a finite number of (ordinary) $\mathcal{L}_{\mathcal{A}}^{\dag}$-definable cells (i.e.\ not parametrized by $RV$-parameters). In the proof, we use Theorem~\ref{SLCD}, Proposition~\ref{descript} and a model-theoretic compactness argument.

\begin{theorem}\label{cell-fin}
For $\mathcal{L}_{\mathcal{A}}^{\dag}$-definable sets
$$ X \subset K^{n} \ \ \text{and} \ \ P \subset X \times RV(K)^{t}, $$
there exists a finite decomposition of $X$ into $\mathcal{L}_{\mathcal{A}}^{\dag}$-definable ordinary cells $C_{k}$, $k=1,\ldots,q$, with centers given by $\mathcal{L}_{\mathcal{A}}^{\dag}$-terms, such that the fibers of $P$ over each twisted box of each $C_{k}$ are constant or, equivalently, the fiber of $P$ over each $\xi \in RV(K)^{t}$ is a union of some twisted boxes from the cells $C_{k}$. Thus, in particular, algebraic and definable closures coincide in the $\mathcal{L}_{\mathcal{A}}^{\dag}$-models of the theory $T_{\mathcal{A}}$.

Furthermore, one can additionally require that each $C_{k}$ is, after some coordinate permutation, a cell of type $(1,\ldots,1,0,\ldots,0)$ with 1-Lipschitz centers
$$ c_{k} = (c_{k,1},\ldots,c_{k,n}), \ \ k=1,\ldots,q. $$
Such cells $C_{k}$ shall be called 0-definable Lipschitz cells.
\end{theorem}

\begin{proof}
We first establish the existence of decomposition into finitely many ordinary cells, proceeding with induction on the dimension of the ambient space $K^{n}$. The case $n=0$ is trivial. Now, assuming the proposition to hold in $K^{n}$, we shall prove it in $K^{n+1}$. Consider thus an $\mathcal{L}_{\mathcal{A}}^{\dag}$-definable subset of $K^{n+1}$. By Theorem~\ref{SLCD}, $X$ is a finite union of $\mathcal{L}_{\mathcal{A}}^{\dag}$-definable parametrized cells $(C_{k},\sigma_{k})$, $k=1,\ldots,p$, preparing $P$. Clearly, we can assume that $X$ is one of those cells; say $X=C$ for a parametrized cell $(C,\sigma)$ with centers
$$ c_{\xi} = (c_{\xi,1,\ldots,\xi,n+1}), \ \ \xi \in \sigma(C). $$

\vspace{1ex}

By Proposition~\ref{descript}, we can assume, at the cost of finer cell decomposition in $K^{n}$, that the centers
$$ c_{\xi,n+1}, \ \ \xi \in \sigma(C), $$
are given piecewise by some $\mathcal{L}_{\mathcal{A}}^{\dag}$-terms. By a routine model-theoretic compactness argument, the union of the graphs of those centers is the union of the graphs
of a finite number of $\mathcal{L}_{\mathcal{A}}^{\dag}$-terms $t_{i}$ on some $\mathcal{L}_{\mathcal{A}}^{\dag}$-definable subsets $D_{i}$ of $K^{n}$, $i=1,\ldots,p$.
If the cell in question is of the type graph, we are done.

\vspace{1ex}

In the other case, we will treat these terms as new centers
$$ c_{1,n+1},\ldots,c_{p,n+1}, $$
which requires further partitioning. For
$$ a \in D := \bigcup_{i=1}^{p} D_{i}, $$
denote by $\Sigma (a)$ the set of those $\xi \in \sigma(C)$ such that $a$ lies in the domain of $c_{\xi,n+1}$, and put
$$ P := \bigcup_{a \in D} \ \{ a \} \times \Sigma (a). $$
By the induction hypothesis, we can find a finite decomposition of $D$ into $\mathcal{L}_{\mathcal{A}}^{\dag}$-definable ordinary cells $E_{1},\ldots,E_{q}$ preparing the set $P$. Then it is not difficult to check that $X$ is the union of some $\mathcal{L}_{\mathcal{A}}^{\dag}$-definable ordinary cells $C_{1},\ldots,C_{q}$ over the cells $E_{1},\ldots,E_{q}$, which concludes the proof of the first conclusion.

\vspace{1ex}

Finally, the second conclusion, concerning Lipschitz cells, can be derived from the first one, similarly as in the classical theorem on Lipschitz cell decomposition.
\end{proof}

Likewise, we have the following version of piecewise Lipschitz continuity established in \cite[Theorem~2.3.1]{C-Fo-L} and \cite[Theorem~5.2.8]{C-H-R}.

\begin{proposition}\label{piece-Lip}
If $f: X \to K$ is an $\mathcal{L}_{\mathcal{A}}^{\dag}$-definable function on a subset $X \subset K^{n}$ which is locally 1-Lipschitz, then there exists an $\mathcal{L}_{\mathcal{A}}^{\dag}$-definable partition of $X$ into finitely many subsets $X_{1}, \ldots, X_{q}$ such that the restriction of $f$ to each $X_{i}$ is 1-Lipschitz.
\end{proposition}

\begin{remark}\label{basic}
The (multi-sorted) basic Henselian functions embrace the additional functions from the language $\mathcal{L}_{\mathcal{A}}^{*}$. Indeed, then both the Henselian witnesses and the roots as functions on $\mathcal{O}_{K}$ are available. The latter is enough since the multiplicative inverse belongs to the language $\mathcal{L}_{\mathcal{A}}$. Therefore all the foregoing results remain valid when $\mathcal{L}_{\mathcal{A}}^{*}$ is the language $\mathcal{L}_{\mathcal{A}}$ augmented by adding the basic Henselian functions. This new meaning of the language $\mathcal{L}_{\mathcal{A}}^{*}$ will be used henceforth.
\end{remark}

We shall still need the concept of $K$-valued {\em Henselian terms} on an (ordinary or parametrized) definable cell $C \subset K^{n}$, defined below by induction on complexity. This concept corresponds with condition (*) from~\cite[Definition~3.2.1]{C-Fo-L}.

\vspace{1ex}

Constants and variables of the valued field sort are Henselian terms on $C$. Further, $\mathcal{L}_{\mathcal{A}}^{\dag}$-terms
$$ s(x) + t(x), \ s(x) \cdot t(x), \ f(t_{1}(x),\ldots,t_{m}(x)) $$
and
$$ h_{m}(t_{0}(x) \ldots, t_{m}(x),\tau(x)), $$
with an analytic function $f$ of the language, are called Henselian on $C$ if the terms
$$ s(x), t(x), t_0(x),\ldots, t_m(x) $$
are Henselian on $C$, $\tau(x)$ is an $RV$-valued term, and the functions
$$ rv(t_0(x)), \ldots, rv(t_m(x)) \ \ \text{and} \ \ \tau(x) $$
are constant on each twisted box of $C$.

\vspace{1ex}

Now, we prove the strengthening of Theorem~\ref{cell-fin} by additionally requiring that the centers of the cells are given by Henselian terms.

\begin{proposition}\label{strong-cell-fin}
For every $\mathcal{L}_{\mathcal{A}}^{\dag}$-definable set $X \subset K^{n}$, there exists a finite decomposition of $X$ into $\mathcal{L}_{\mathcal{A}}^{\dag}$-definable ordinary cells $C_{k}$, $k=1,\ldots,q$, with centers given by Henselian terms.

Furthermore, one can additionally require that each $C_{k}$ is, after some coordinate permutation, a cell of type $(1,\ldots,1,0,\ldots,0)$ with 1-Lipschitz centers
$$ c_{k} = (c_{k,1},\ldots,c_{k,n}), \ \ k=1,\ldots,q. $$
\end{proposition}

\begin{proof}
This follows directly by induction on the complexity of the terms that define the centers from the conclusion of the proposition, and successive cell decompositions preparing the $RV$-parametrized sets that correspond to their subterms.
\end{proof}

Hence we immediately obtain the following result on the strong Lipschitz para\-metrization of definable sets.

\begin{corollary}\label{Lip-parameter}
Every $\mathcal{L}_{\mathcal{A}}^{\dag}$-definable set $X \subset K^{n}$ is a finite union of the images of $\mathcal{L}_{\mathcal{A}}^{\dag}$-definable cells with zero centers under 1-Lipschitz map given by Henselian terms.   
\end{corollary}


\section{Yomdin--Gromov parametrizations}

We begin with a brief historical background. The concept of Yomdin--Gromov para\-metrizations goes back to their papers~\cite{Yo-0,Yo,Grom} devoted to dynamical systems (entropy problem). It is of great importance that their fundamental result provides the minimum number $s(r,X)$ of $\mathcal{C}^{r}$-smooth charts for a given bounded semialgebraic subset $X$ of $\mathbb{R}^{n}$ only in terms of some combinatorial data. This allowed them to compare the topological entropy and the ``homological size'' of $\mathcal{C}^{r}$-smooth maps.

\vspace{1ex}

Let us mention that the original proof of the famous Yomdin--Gromov algebraic lemma was not entirely complete. Its generalizations to the o-minimal settings were initiated by Pila--Wilkie~\cite{PW}, allowing them to establish their celebrated counting theorem about rational points on the transcendental part of definable sets. They roughly followed Gromov's line of reasoning, but with much more involved technical details. This remarkable application in Diophantine geometry was a continuation of the previous research, including~\cite{BP,P}. For a survey on Diophantine applications of the Pila–Wilkie theorem, see e.g.~\cite{Sc}.

\vspace{1ex}

Recently, some crucial, sharper estimates for the number $s(r,X)$ and the complexity of semialgebraic $\mathcal{C}^{r}$-smooth charts of such a subset $X$ of $\mathbb{R}^{n}$ are given by Binyamini--Novikov~\cite{BN-0}.
While the former estimate
$$ s(r,X) \leq \mathrm{poly}\, (\beta) \cdot r^{m} $$
depends polynomially on the smoothness order $r$ and the complexity $\beta$ of equations defining $X$, and exponentially on the dimension $m$ of $X$, the latter is polynomial in $r$ and $\beta$. 

\vspace{1ex}

Notice, however, that the notion of complexity is not available in the general o-minimal settings, and must be replaced by uniformity over a definable family $X$. Again, understanding the behaviour of the number $s(r,X)$ of charts with respect to the definable family $X$ and the smoothness order $r$ is of great importance.

\vspace{1ex}

Afterwards, some other approaches were given in the papers~\cite{Paw,BN-0,BN,C-P-W}. Let us emphasize that Cluckers--Pila-Wilkie~\cite{C-P-W} made a significant progress for real, subanalytic sets (and also in the real analytic structure with exponentiation) by estimating the number $s(r,X)$ for a subanalytic family $X$ polynomially with respect to $r$.

\vspace{1ex}

The recent papers~\cite{C-Com-L,C-Fo-L} provide non-Archimedean, uniform Yomdin--Gromov parametrizations, together with estimates for the number of charts, over the equicharacteristic zero, complete, discretely valued fields $K$, possibly with some restricted analytic function symbols coming from~\cite[Example~4.4(1)]{C-Lip-0}, and over the local fields $K$, i.e.\ finite field extensions of $\mathbb{Q}_{p}$ or $\mathbb{F}_{p}((t))$; the latter fields are locally compact. Those parametrizations and estimates are used to estimate the number of integer and rational points of bounded height in algebraic or analytic varieties over $K$.

\vspace{1ex}

In this section, we investigate Yomdin--Gromov parametrizations in the general case of arbitrary equicharacteristic zero valued fields with analytic structure (Theorem~\ref{strong-r}), i.e.\
in the models $K$ of the theory $T_{\mathcal{A}}$.
Our goal is to prove that definable subsets of $\mathcal{O}_{K}^{n}$ admit, uniformly in definable families, finitely many strong Yomdin--Gromov $T^{1}$-parametrizations, and to estimate their number.

\vspace{1ex}

Eventually, the idea from~\cite{C-Com-L,C-Fo-L} of obtaining $T^{r}$-parametrizations from any given strong $T^{1}$-parametrizations via precomposing it with sufficiently large power functions, can be repeated almost verbatim in the general case under study.

\vspace{1ex}

Therefore finite $T^{r}$-parametrizations for any positive integer $r$ are available, provided that some prime invariant $[p]RV(K)$ (cf.~\cite{Rob-Zak}) of the group $RV(K)$ is finite.

\vspace{1ex}

Now we introduce, following the paper~\cite{C-Fo-L}, the basic terminology concerning $T_{r}$-parametrizations.
For any valued field $K$, a subset $P$ of $\mathcal{O}_{K}^{m}$ and a positive integer $r$, a function
$$ f = (f_{1},\ldots,f_{n}): P \to \mathcal{O}_{K}^{n} $$
is said to satisfy $T_{r}$-{\em approximation} if $P$ is open in $\mathcal{O}_{K}^{m}$, and for each $a \in P$ there is an $n$-tuple $T_{f,a}^{<r}$ of (unique) polynomials with coefficients in $\mathcal{O}_{K}$ of degree $<r$ such that
$$ \left| f(x) - T_{f,a}^{<r}(x) \right| < |x-a|^{r} \ \ \text{for all} \ \ x \in P. $$
If $f$ is of class $C^{r}$ in the sense of~\cite{BGN}, they are just the Taylor polynomials of $f$ at $a$ of order $r$ by Taylor's formula (see e.g.~\cite[Theorem~5.1]{BGN}). Clearly, $f$ satisfies $T_{1}$-approximation if and only if $f$ is 1-Lipschitz.

\vspace{1ex}

Let us mention that the set of points where a definable function $f: P \to K$ is not of class $C^{r}$ is definable and nowhere dense in $P$ (cf.~\cite[Theorem~5.1.5]{C-H-R}).

\vspace{1ex}

We say that a family $(f_{i})_{i \in I}$ of functions $f_{i}: P_{i} \to \mathcal{O}_{K}^{n}$ is a $T_{r}$-parametrization of a subset $X$ in $ \mathcal{O}_{K}^{n}$ if
$$ X = \bigcup_{i \in I} \; f_{i}(P_{i}) $$
and each $f_{i}$ satisfies $T_{r}$-approximation.

\vspace{1ex}

Observe that Cauchy’s estimates are key for the calculations related to power substitution.
Therefore, the concept of $T^{1}$-parametrization should be strengthened to be able to use Cauchy's estimates for the involved Henselian terms defined on open polydiscs. Hence the following two cases, affecting the definition of strong approximation below, are encountered:

\vspace{1ex}

(V1) the value group $vK$ is densely ordered, i.e.\ it has no smallest positive element $1$;

(V2) conversely, $vK$ has a smallest positive element $1$.

\vspace{1ex}

Let $K$ be a model of the theory $T_{\mathcal{A}}$. We say that the function $f$ as above satisfies strong $T_{1}$-approximation if $P$ is an open cell with zero centers, and $f$ is given on $P$ by Henselian terms which are 1-Lipschitz on $P$.

\vspace{1ex}

Moreover, in case (V2), those terms are required to be 1-Lipschitz also on the associated box $B_{as}$ to each box $B$ of $P$ (cf.~Lemma~\ref{hen} ff.). Here $B_{as}$ is the box over the algebraic closure $K_{alg}$ of $L$ defined by the same condition as $B$ over $K$.

\vspace{1ex}

As before, we say that a family $(f_{i})_{i \in I}$ of functions $f_{i}: P_{i} \to \mathcal{O}_{K}^{n}$ is a strong $T_{1}$-parametrization of a subset $X$ in $ \mathcal{O}_{K}^{n}$ if
$$ X = \bigcup_{i \in I} \; f_{i}(P_{i}) $$
and each $f_{i}$ satisfies strong $T_{1}$-approximation.

\vspace{1ex}

We can now formulate our main result on strong $T_{1}$-parametrization for models $K$ of the theory $T_{\mathcal{A}}$.

\begin{theorem}\label{strong-1}
Consider an $\mathcal{L}_{\mathcal{A}}$-definable subset
$$ X \subset K^{k} \times \mathcal{O}_{K}^{n}, $$
regarded as a definable family
$$ X_{w} = \{ x \in K^{n}: \ (w,x) \in X \}, \ \ w \in W \subset K^{m}. $$
Suppose that the set $X_{w}$ is of dimension $m$ for each $w \in W$. Then there exist a finite set $I$ of cardinality $s=s(X)$, and a definable family
$$ f = (f_{w,i})_{(w,i) \in W \times I} $$
of definable functions
$$ f_{w,i}: P_{w,i} \to X_{w} \ \ \text{with} \ \ P_{w,i} \subset \mathcal{O}_{K}^{m},  $$
such that $(f_{w,i})_{i \in I}$ is a strong $T_{1}$-parametrization of $X_{w}$ for each $w \in W$.
\end{theorem}

\begin{proof}
In case (V1), the conclusion for a single set $X$ follows immediately from Proposition~\ref{strong-cell-fin}.

\vspace{1ex}

Case (V2) is far more difficult, and involves the concept of strong stratification. The following version for a single set $X$ can be proven by induction on the dimension $m$ of $X$. This will be done, as outlined below, by means of strong stratification, piecewise Lipschitz continuity and cell decomposition with centers given by 1-Lipschitz Henselian terms.

\vspace{1ex}

\begin{em}
Consider an $\mathcal{L}_{\mathcal{A}}$-definable subset $X$ of $\mathcal{O}_{K}^{n}$. Then there exist a finite set $I$ and a definable family
$f = (f_{i})_{i} \in I$ of definable functions
$$ f_{i}: P_{i} \to X \ \ \text{with} \ \ P_{i} \subset \mathcal{O}_{K}^{m} $$
such that $(f_{i})_{i \in I}$ is a strong $T_{1}$-parametrization of $X$.
\end{em}

\vspace{1ex}

Indeed, by virtue of Theorem~\ref{strat-strong} and Remark~\ref{strat-uniform}, there exists the strong stratification $S_{k} = \sigma_{k}(W_{k})$ with $W_{k} \subset N_{k}$, $k=1,\ldots,s$, which is the restriction of the one over the algebraic closure $K_{alg}$ with data $\sigma_{k}'$, $W_{k}'$. We can of course assume that $X$ is one of the sets $S_{k}$, say $X = S$ where $S = \sigma(W)$, $\sigma =\sigma'$ and $W :=W_{k} \subset N := N_{k}$, and that
$$  m = \dim X = \dim W = \dim S . $$
Moreover, the strong analytic submanifold $W$ is the trace of an $\mathcal{L}_{\mathcal{A}}$-definable T-analytic submanifold $M'$ of $(K_{alg}^{\circ})^{n} \times \mathbb{P}^{N}(K_{alg})$ on an open $\mathcal{L}_{\mathcal{A}}$-definable subset $\Omega := \Omega_{k}$ of $N \subset (K^{\circ})^{n} \times \mathbb{P}^{N}(K)$.

\vspace{1ex}

Furthermore, we can assume, via a finite $\mathcal{L}_{\mathcal{A}}$-definable partitioning and permutation of coordinates, that $M'$ is a strong analytic submanifold in $\mathcal{O}_{K_{alg}}^{n+N}$ whose projection into $K_{alg}^{m}$ is a finite-to-one, and which is locally the graph of a 1-Lipschitz map.

\vspace{1ex}

Making use of Propositions~\ref{strong-cell-fin} and~\ref{piece-Lip}, throwing away pieces of lower dimension and treating them by induction, we can assume that $M'$ is the graph of a tuple of 1-Lipschitz Henselian $\mathcal{L}_{\mathcal{A}}^{*}$-terms over an open cell $P'$ in $K_{alg}^{m}$ with centers given by 1-Lipschitz Henselian terms, and that $W = M' \cap \Omega$. This uniquely determines the $T_{1}$-parametrization $\tau'$ of $M'$ given by Henselian $\mathcal{L}_{\mathcal{A}}^{*}$-terms defined on the open cell $P''$ with zero centers, which is associated to the cell $P'$.

\vspace{1ex}

Then it is easy to check that its restriction $\tau$ to the pre-image of $M' \cap \Omega$ is a strong $T_{1}$-parametrization of $W$. Since the blowup
$$ \sigma: N \cap \mathcal{O}_{K}^{n+N} \to \mathcal{O}_{K}^{n} $$
is the projection onto the first $n$ coordinates, the composite map $\sigma \circ \tau$ is a strong $T_{1}$-parametrization of $S = \sigma(W)$, which is the desired result.

\vspace{1ex}

Finally, the family version can be obtained by model-theoretic compactness as shown below.
Suppose, just for the sake of simplicity, that the Weierstrass system $\mathcal{A} = \mathcal{A}(K)$. Given any model $L$ elementarily equivalent to $K$, one can expand the analytic structure on $L$ to $\mathcal{A}(L)$ by adding the parameters from $L$ or, more generally, to $\mathcal{A}(F)$ by adding the parameters from any subfield $F$ of $L$ (cf.~\cite[Remark~4.5.8]{C-Lip-0}); in particular, from the subfield $K(a)$ for any $a \in L$. More precisely, the expanded rings $A_{m,n}(F)$ are given by the formula
$$ \bigcup_{m',n' \in \mathbb{N}} \; \left\{ f(\xi,c,\rho, d): \ c \in (F^{\circ})^{m'}, \ d \in (F^{\circ\circ})^{n'}, \ f \in A_{m+m',n+n'} \right\}. $$
Thus we can use a routine model-theoretic compactness argument.

\end{proof}

\begin{remark}
The above theorem on strong $T_{1}$-parametrization can be easily strengthened by requiring that, for every $w \in W$, each map $f_{w,i}$ is injective, and the images $f_{w,i}(P_{w,i})$, $i \in I$, are pairwise disjoint.
\end{remark}


Now we return to the basic Henselian functions
$$ h_{m} : K^{m+1} \times RV(K) \to K. $$
Given points $a_{0},\ldots, a_{m},b \in K$, $b \neq 0$, consider the open box
$$ B := \prod_{i=1}^{m} a_{i} \cdot (1 + \mathfrak{m}). $$
The following lemma holds (cf.~\cite[Lemma~6.3.11]{C-Lip-0}).

\begin{lemma}\label{hen}
Under the above assumptions, the restriction
$$  h: B \times \{ rv(b) \} \to b \cdot (1 + \mathfrak{m}) $$
of the function $h_{m}$ is given by a power series in
$$ \mathcal{O}^{\dag}(B) := K \otimes_{\mathcal{O}_{K}} \mathcal{O}(B); $$
in our convention, $h \in \mathcal{O}^{\dag}(B)$.
\end{lemma}

Now observe that the supremum norm of any function
$f \in A^{\dag}_{m,n}$ on $(K^{\circ}_{alg})^{m} \times (K^{\circ \circ}_{alg})^{n}$
equals the Gauss norm of $f$ (cf.~\cite[Remark~5.2.8]{C-Lip-0}). Clearly, this remains true for the supremum norm on $(K^{\circ})^{m} \times (K^{\circ \circ})^{n}$ whenever the residue field $Kv$ is infinite and the value group $vK$ is densely ordered.
Hence the classical Cauchy's estimates hold on any open box $B_{as}$ for each $f \in \mathcal{O}^{\dag}(B_{as})$; and they hold on any open box $B$ for each $f \in \mathcal{O}^{\dag}(B)$ under the above condition on the residue field and value group. This is why the concept of strong $T_{1}$-approximation has been introduced.

\vspace{1ex}

Moreover, as already indicated, a $T_{r}$-approximation can be obtained from any given $T_{1}$-approximation via precomposing with the $r$-th power functions, which relies heavily on Cauchy's estimates.

\vspace{1ex}

To achieve finite parametrizations, similarly as in~\cite{C-Fo-L}, one must require that the $r$-th congruence invariant $b_{r} = [r]RV(K)$ of the group $RV(K)$ is finite; or at least, that the prime invariant $[p]RV(K)$ of $RV(K)$ is finite for a prime number $p$ (cf.~\cite{Rob-Zak}). (Otherwise one can achieve only parametrizations in the infinite number of cosets of the $r$-th powers in the group $RV(K)$.) In the latter case, precomposing with some power $p^{l} > r$ is sufficient. In this fashion, Theorem~\ref{strong-1} yields the following main result on $T_{r}$-parametrizations.

\begin{theorem}\label{strong-r}
Consider an $\mathcal{L}_{\mathcal{A}}$-definable subset
$$ X \subset K^{k} \times \mathcal{O}_{K}^{n}, $$
regarded as a definable family
$$ X_{w} = \{ x \in K^{n}: \ (w,x) \in X \}, \ \ w \in W \subset K^{k}, $$
and a positive integer $r$. Suppose that the set $X_{w}$ is of dimension $m$ for each $w \in W$ and that the congruence invariant $b_{r} = [r]RV(K)$ is finite. Then there exist a positive integer $s = s(X)$, a finite set $I$ of cardinality
$$ s(r,X) \leq s \cdot b_{r}^{m}, $$
and a definable family
$$ f = (f_{w,i})_{(w,i) \in W \times I} $$
of definable functions
$$ f_{w,i}: P_{w,i} \to X_{w} \ \ \text{with} \ \ P_{w,i} \subset \mathcal{O}_{K}^{m},  $$
such that $(f_{w,i})_{i \in I}$ is a $T_{r}$-parametrization of $X_{w}$ for each $w \in W$.
\end{theorem}

\begin{remark}
If a prime invariant $[p]RV(K)$ is finite, then the uniform estimate of the number of $T_{r}$-parametrizations is given by
$$ s(r,X) \leq s \cdot ([p]RV(K))^{lm}, $$
where $l$ is the smallest integer with $r \leq p^{l}$.
\end{remark}

\vspace{3ex}

\begin{small}

Institute of Mathematics

Faculty of Mathematics and Computer Science

Jagiellonian University

ul.~Profesora S.\ \L{}ojasiewicza 6

30-348 Krak\'{o}w, Poland

{\em E-mail address: nowak@im.uj.edu.pl}
\end{small}


\vspace{2ex}


\begin{thebibliography}{99}

\bibitem{Bas}
S.A.~Basarab, {\em Relative elimination of quantifiers for Henselian valued fields}, Ann.\ Pure Appl.\ Logic {\bf 53} (1991), 51--74.

\bibitem{BGN}
W.~Bertram, H.~Gl\"{o}ckner, K.-H.~Noeb, {\em Differential calculus over general base fields and rings}, Expo.\ Math.\ {\bf 22} (2004), 213--282.

\bibitem{BN-0}
G.~Binyamini, D.~Novikov, {\em Complex cellular structures}, Ann.\ Math.\  {\bf 190} (2019), 145--248.

\bibitem{BN}
G.~Binyamini, D.~Novikov, {\em The Yomdin--Gromov algebraic lemma revisited}, Arnold Math.\ J.\ {\bf 7} (2021), 419--430.

\bibitem{BM-0}
E.~Bierstone, P.D.~Milman, {\em Semianalytic and subanalytic sets}, Publ.\ Math.\ I.H.\'{E}.S.\ {\bf 67} (1988), 5-42.

\bibitem{BM}
E.~Bierstone, P.D.~Milman, {\em Canonical desingularization in
characteristic zero by blowing up the maximum strata of a local
invariant\/}, Inventiones Math.\ {\bf 128} (1997), 207--302.

\bibitem{BP}
E.~Bombieri, J.~Pila, {\em The number of integral points on arcs and ovals}, Duke Math.\ J.\ {\bf 59} (1989), 337--357.

\bibitem{BGR}
S.~Bosch, U.~G\"{u}ntzer, R.~Remmert, {\em Non-Archimedian
Analysis: a systematic approach to rigid analytic geometry\/},
Grundlehren der math.\ Wiss.\ {\bf 261}, Springer-Verlag, Berlin,
1984.

\bibitem{C-Com-L}
R.~Cluckers, G.~Comte, F.~Loeser, {\em Non-Archimedean
Yomdin--Gromov parametrizations and points of bounded height},
Forum Math.\ Pi, {\bf 3} (2015), e5.

\bibitem{C-Fo-L}
R.~Cluckers, A.~Forey, F.~Loeser, {\em Uniform Yomdin--Gromov para\-metrizations and points of bounded height in valued firelds}, Algebra Number Theory {\bf 14} (2020), 1423--1456.

\bibitem{C-H-R}
R.~Cluckers, I.~Halupczok, S.~Rideau, {\em Hensel minimality I}, Forum Math.\ Pi, {\bf 10} (2022), e11.

\bibitem{C-H-R-1}
R.~Cluckers, I.~Halupczok, S.~Rideau, F.~Vermeulen, {\em Hensel minimality II: Mixed characteristic and a diophantine application}, Forum Math.\ Sigma, {\bf 11} (2023), e89.


\bibitem{C-Lip-R}
R.~Cluckers, L.~Lipshitz, Z.~Robinson, \emph{Analytic cell
decomposition and analytic motivic integration}, Ann.\ Sci.\ École
Norm.\ Sup.\ (4) {\bf 39} (2006), 535--568.

\bibitem{C-Lip-0}
R.~Cluckers, L.~Lipshitz, {\em Fields with analytic structure\/},
J.\ Eur.\ Math.\ Soc.\ {\bf 13} (2011), 1147--1223.

\bibitem{C-Lip}
R.~Cluckers, L.~Lipshitz, {\em Strictly convergent analytic
structures\/}, J.\ Eur.\ Math.\ Soc.\ {\bf 19} (2017), 107--149.

\bibitem{C-P-W}
R.~Cluckers, J.~Pila, A.~Wilkie, {\em Uniform parametrization of subanalytic sets and Diophantine applications}, Ann.\ Sci.\ \'{E}cole Norm.\ Sup.\ (4) {\bf 53} (2020), 1-42.

\bibitem{Co}
P.J.~Cohen, {\em Decision procedures for real and p-adic fields}, Comm.\ Pure Appl.\ Math.\ {\bf 22} (1969), 131–151.

\bibitem{De}
J.~Denef, {\em p-adic semi-algebraic sets and cell decomposition}, J.\ Reine Angew.\ Math.\ {\bf 369} (1986), 154–166.


\bibitem{E-P}
A.J.~Engler, A.~Prestel, {\em Valued Fields}, Springer-Verlag, Berlin, Heidelberg, 2005.

\bibitem{Grom}
M.~Gromov, {\em Entropy, homology and semialgebraic geometry}, S\'{e}m.\ Bourbaki, Ast\'{e}risque {\bf 145--146} (1987), 225--240.


\bibitem{Kei}
H.J.~Keisler, {\em Fundamentals of model theory}; In: {\em Handbook of Mathematical Logic}, 47--103, Ed.\ J.~Barwise, Studies in Logic and the Foundation of Math., Vol.~50, Elsevier, Amsterdam, 1977.

\bibitem{Paw}
B.~Kocel-Cynk, W.~Pawłucki, A.~Valette, {\em $C^{p}$-parametrization in o-minimal structures}. Can.\ Math.\ Bull.\ {\bf 62} (2019), 99–108.

\bibitem{K-N}
J.~Koll{\'a}r, K.~Nowak, {\em Continuous rational functions on
real and $p$-adic varieties\/}, Math. Zeitschrift {\bf 279}
(2015), 85--97.

\bibitem{LR-separated}
L.~Lipshitz, Z.~Robinson, {\em Rings of separated power series}, Ast\'{e}risque {\bf 264} (2000), 3--108.

\bibitem{LR-subanalytic}
L.~Lipshitz, Z.~Robinson, {\em Model completeness and subanalytic sets}, Ast\'{e}risque {\bf 264} (2000), 109--126.

\bibitem{LR-strat}
L.~Lipshitz, Z.~Robinson, {\em Dimension theory and smooth stratification of rigid subanalytic sets}; In: Logic Colloquium ’98, Lect.\ Notes Logic {\bf 13} (2000), Assoc.\ Symbolic Logic, 302--315.

\bibitem{LR-uniform}
L.~Lipshitz, Z.~Robinson, {\em Uniform properties of rigid subanalytic sets}, Trans.\ Amer.\ Math.\ Soc.\ {\bf 357} (2005), 4349--4377.

\bibitem{Now-Sel}
K.J.~Nowak, {\em Some results of algebraic geometry over Henselian
rank one valued fields\/},  Sel.\ Math.\ New Ser.\ {\bf 23}
(2017), 455--495.

\bibitem{Now-Sym}
K.J.~Nowak, {\em Definable transformation to normal crossings over Henselian fields
with separated analytic structure}, Symmetry {\bf 11} (7) (2019), 934.

\bibitem{Now-Sing}
K.J.~Nowak, {\em A closedness theorem and applications in geometry
of rational points over Henselian valued fields}, J.\ Singul.\ {\bf 21} (2020), 212--233.

\bibitem{Now-APAL}
K.J.~Nowak, {\em Tame topology in Hensel minimal structures}, Ann.\ Pure Appl.\ Logic {\bf 176} (2025), 103540.

\bibitem{Nub}
H.~N\"{u}bling, {\em Adding Skolem functions to simple theories}, Arch.\ Math.\ Logic {\bf 43} (2004), 359--370.

\bibitem{Pa1}
J.~Pas, {\em Uniform p-adic cell decomposition and local zeta
functions\/}, J.\ Reine Angew.\ Math.\ {\bf 399} (1989), 137--172.

\bibitem{P}
J.~Pila, {\em Density of integral and rational points on varieties}, Ast\'{e}risque {\bf 228} (1995), 183--187.

\bibitem{PW}
J.~Pila, A.J.~Wilkie, {\em The rational points of a definable set}, Duke Math.\ J.\ {\bf 133} (2006), 591--616.

\bibitem{Rob-Zak}
A.~Robinson, E.~Zakon, {\em Elementary properties of ordered abelian groups}. Trans.\ Amer.\ Math.\ Soc.\ {\bf 96} (1960), 222--236.

\bibitem{Sc}
T.~Scanlon, {\em O-minimality as an approach to the Andr\'{e}--Oort conjecture}, Panoramas \& Synth\`{e}ses {\bf 52} (2017), 111--165.

\bibitem{Se}
J.P.~Serre, {\em Lie Algebras and Lie Groups}; Lect.\ Notes Math., Vol.\ 1500, Springer, 1992.

\bibitem{Tem-2}
M.~Temkin, {\em Functorial desingularization over $\mathbb{Q}$:
boundaries and the embedded case}, Israel J.\ Math. {\bf 224} (2018), 455--504.

\bibitem{Ver}
F.~Vermeulen, {\em Counting rational points on transendental curves in valued fields}, arXiv:2506.19411 [math.NT].

\bibitem{Yo-0}
Y.~Yomdin, {\em Volume growth and entropy}, Israel J.\ Math.\ {\bf 57} (1987), 285--300.

\bibitem{Yo}
Y.~Yomdin, {\em $C^{k}$-resolution of semialgebraic mappings. Addendum to Volume growth and entropy}, Israel J.\ Math.\ {\bf 57} (1987), 301--317.










\end{thebibliography}
\end{document}